\documentclass[oneside,reqno,10pt]{amsart}
\usepackage{graphicx,amsfonts,amssymb,amsmath,amsthm,url,comment}
\usepackage{color}
\usepackage[usenames,dvipsnames]{xcolor}
\usepackage[normalem]{ulem}
\usepackage{pdfsync}

\usepackage[hyperfootnotes=false, colorlinks, citecolor=	 RoyalBlue, urlcolor=blue, linkcolor=blue         ]{hyperref}

\theoremstyle{plain} 

\newtheorem{theorem}             {Theorem}  
\newtheorem{lemma}      [theorem]{Lemma}

\newtheorem{conjecture} [theorem]{Conjecture}
\newtheorem{question}   [theorem]{Question}

\theoremstyle{definition}
\newtheorem{definition} [theorem]{Definition}

\theoremstyle{remark}
\newtheorem{remark}            [theorem]  {Remark}

\renewcommand{\geq}{\geqslant}
\renewcommand{\leq}{\leqslant}

\DeclareMathOperator{\sgn}{sgn}
\DeclareMathOperator{\SL}{SL}

\DeclareMathOperator{\GL}{GL}

\DeclareMathOperator{\prim}{prim}

\DeclareMathOperator{\ad}{ad}

\DeclareMathOperator{\disc}{disc}

\DeclareMathOperator{\Hom}{Hom}
\DeclareMathOperator{\End}{End}

\def\eps{\varepsilon}

\DeclareMathOperator{\PGL}{PGL}

\DeclareMathOperator{\RS}{RS}

\DeclareMathOperator{\diag}{diag}

\DeclareMathOperator{\pr}{pr}

\DeclareMathOperator{\reg}{reg}

\DeclareMathOperator{\vol}{vol}
\DeclareMathOperator{\nr}{nr}

\DeclareMathOperator{\tr}{tr}

\def\eps{\varepsilon}

\begin{document}

\title{Microlocal lifts and quantum unique ergodicity on $\GL_2(\mathbb{Q}_p)$}

\author{Paul D. Nelson}
\email{paul.nelson@math.ethz.ch}
\address{ETH Zurich, Department of Mathematics, R{\"a}mistrasse 101, CH-8092, Zurich, Switzerland}
\begin{abstract}
  We prove that arithmetic quantum unique ergodicity holds on
  compact arithmetic quotients of $\GL_2(\mathbb{Q}_p)$ for
  automorphic forms belonging to the principal series.
  We interpret this conclusion in terms
  of the equidistribution of eigenfunctions
  on covers of a fixed regular graph
  or along nested sequences of regular graphs.

  Our results are the first of their kind on any $p$-adic
  arithmetic quotient.  They may be understood as
  analogues of Lindenstrauss's theorem on the equidistribution
  of Maass forms on a compact
  arithmetic surface.  The new ingredients here include the
  introduction of a representation-theoretic notion of
  ``$p$-adic microlocal lifts'' with favorable properties, such
  as diagonal invariance of limit measures; the proof of
  positive entropy of limit measures in a $p$-adic aspect,
  following the method of Bourgain--Lindenstrauss; and some
  analysis of local Rankin--Selberg integrals involving
  the microlocal lifts introduced here as well as classical
  newvectors.  An important input is a measure-classification
  result of Einsiedler--Lindenstrauss.
\end{abstract}
\maketitle
\setcounter{tocdepth}{1}
\tableofcontents

\section{Introduction\label{sec:intro}}
\label{sec-1}

\subsection{Overview}\label{sec:overview}
Let $p$ be a prime number.  This article is concerned with the
limiting behavior of eigenfunctions on compact arithmetic
quotients
of the group $G := \GL_2(\mathbb{Q}_p)$.
A rich class of such
quotients is parametrized by the definite quaternion algebras
$B$ 
over $\mathbb{Q}$ that split at $p$.
A maximal order $R$ in such
an algebra
and an embedding $B \hookrightarrow M_2(\mathbb{Q}_p)$
give rise to a discrete cocompact subgroup
$\Gamma := R[1/p]^\times$ of $G$.
Fix one such $\Gamma$.
The corresponding arithmetic quotient
$\mathbf{X} := \Gamma \backslash G$ is then compact; in
interpreting this, it may help to note that the center of
$\Gamma$ is the discrete cocompact subgroup
$\mathbb{Z}[1/p]^\times$ of $\mathbb{Q}_p^\times$.

The space $\mathbf{X}$ is a $p$-adic analogue of the cotangent
bundle of an arithmetic hyperbolic surface, such as the modular
surface $\SL_2(\mathbb{Z}) \backslash \mathbb{H}$.  It comes
with commuting families of Hecke correspondences $T_\ell$
indexed by the primes $\ell \neq p$ (see \S\ref{sec:definition-t_n}).
To zeroth approximation, the space $\mathbf{X}$ is modelled by
its minimal quotient
$\mathbf{Y} := \mathbf{X}/K = \Gamma \backslash G / K$ by the
maximal compact subgroup $K := \GL_2(\mathbb{Z}_p)$ of $G$.
That quotient $\mathbf{Y}$ comes
with an additional Hecke correspondence
$T_p$.
To simplify the exposition
of \S\ref{sec:overview}, it will convenient to
assume that
\begin{equation}\label{eq:torsion-free-assumption-yay}
  \text{ ( the torsion subgroup of $\Gamma$ ) } = \{\pm 1\}.
\end{equation}
Then $\mathbf{Y}$ 
may be safely regarded as an undirected $(p+1)$-regular finite
multigraph
(see \cite{MR580949, MR1954121}, \cite[\S8]{MR2195133})
whose adjacency graph is $T_p$.
The simplifying assumption \eqref{eq:torsion-free-assumption-yay} holds when the underlying
quaternion algebra has discriminant (say) $73$, in which case 
the graph $(\mathbf{Y},T_p)$ may be depicted
as follows when $p=2,3$:\footnote{
  The images were produced
  using the ``Graph'' and ``BrandtModule'' functions in SAGE \cite{sage2015}.}
\begin{center}\label{pic:73-2}
  \includegraphics[width=4cm,height=4cm]{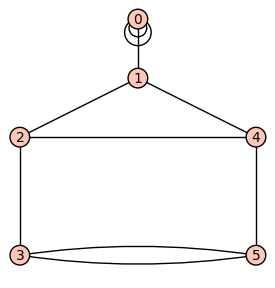}
  \includegraphics[width=4cm,height=4cm]{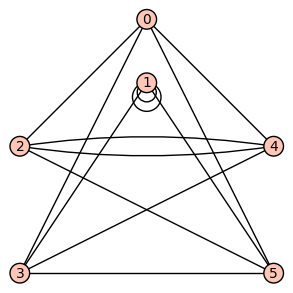}
\end{center}
Such graphs and their eigenfunctions appear naturally in several
contexts, and have been extensively studied since the pioneering
work of Brandt and Eichler \cite{MR0017775, MR0080767}; they specialize to the $p$-isogeny
graphs of elliptic curves in finite characteristic 
\cite[\S2]{MR894322},
provide an important tool for constructing spaces of modular forms
\cite{MR579066},
and their remarkable expansion properties
have been studied and applied in computer science following
\cite{MR963118}.

To study the space $\mathbf{X}$ at a finer
resolution
than that of 
its minimal quotient
$\mathbf{Y}$,
we introduce for each pair of integers $m,m'$
the notation $m..m' := \{m, m+1, \dotsc, m'\}$
and set
\begin{equation}\label{eq:example-path-m-mp}
  \mathbf{Y}_{m..m'}
  :=
  \left\{
    \begin{split}
      &\text{ non-backtracking paths } 
      x = (x_{m} \rightarrow x_{m+1} \rightarrow \dotsb \rightarrow
      x_{m'}) \\
      &\quad \quad \quad
      \text{ indexed by } m..m' 
      \text{ on the graph } (\mathbf{Y} ,T_p)
    \end{split}
  \right\}.
\end{equation}
We will recall in Definition \ref{defn:group-theoretic-paths} the 
standard group-theoretic realization
of $\mathbf{Y}_{m..m'}$
as a quotient of $\mathbf{X}$.
We may and shall identify $\mathbf{Y}_{0..0}$ with $\mathbf{Y}$.
For $m..m' \supseteq n..n'$,
we define compatible surjections
$\mathbf{Y}_{m..m'} \rightarrow \mathbf{Y}_{n..n'}$
by forgetting part of the path.
For example,
if $N \geq 0$,
then the map $\mathbf{Y}_{-N..N} \rightarrow
\mathbf{Y}_{0..0} = \mathbf{Y}$
sends a path $x$
as in \eqref{eq:example-path-m-mp}
to its central vertex $x_0$.
We define $L^2(\mathbf{Y}_{m..m'})$ with
respect to the normalized counting measure,
so that the maps
$\mathbf{Y}_{m..m'} \rightarrow \mathbf{Y}_{n..n'}$
are measure-preserving.

We wish to study the
asymptotic behavior of ``eigenfunctions''
in $L^2(\mathbf{Y}_{m..m'})$
as $|m-m'| \rightarrow \infty$.
From the arithmetic perspective,
there is a distinguished
collection of such eigenfunctions,
whose definition
is analogous to that of the set of normalized
classical holomorphic newforms of some given weight and level:

\begin{definition}[Newvectors]\label{defn:newvectors-super-classicla}
  Let by $\mathcal{F}_{m..m'} \subseteq
  L^2(\mathbf{Y}_{m..m'})$
  be
  an orthonormal basis
  for the space of functions $\varphi : \mathbf{Y}_{m..m'} \rightarrow
  \mathbb{C}$
  satisfying the following conditions:
  \begin{enumerate}
  \item the pullback of $\varphi$
    to $\mathbf{X} = \Gamma \backslash G$
    generates an irreducible representation
    of $G = \GL_2(\mathbb{Q}_p)$
    under the right translation action.
  \item $\varphi$ is an eigenfunction of the Hecke operator
    $T_\ell$ (see \S\ref{sec:definition-t_n})
    for all primes $\ell \neq p$.
  \item $\varphi$ is orthogonal to pullbacks
    from $\mathbf{Y}_{n..n'}$
    whenever $n..n' \subsetneq m..m'$.
  \end{enumerate}
  It is known that
  $|\mathcal{F}_{m..m'}| \asymp |\mathbf{Y}_{m..m'}|
  \asymp p^{|m-m'|}$
  for $|m - m'|$ sufficiently large.
\end{definition}
To simplify the exposition of \S\ref{sec:overview},
we focus on the symmetric intervals $-N..N$.
Fix $n \in \mathbb{Z}_{\geq 0}$.
Let $N \geq n$ be an integral parameter tending off to $\infty$.
Denote by $\pr : \mathbf{Y}_{-N..N} \twoheadrightarrow
\mathbf{Y}_{-n..n}$
the natural surjection.
For $\varphi \in \mathcal{F}_{-N..N}$,
we may define a probability measure
$\mu_\varphi$ on $\mathbf{Y}_{-n..n}$
by setting
\[
\mu_\varphi(E)
:=
\frac{1}{|\mathbf{Y}_{-N..N}|}
\sum_{\substack{x \in \mathbf{Y}_{-N..N} :
    \\
    \pr(x) \in  E
  }}
|\varphi|^2(x).
\]
For example,
in the instructive special  case $n=0$,
the measures $\mu_\varphi$
live on the base graph $\mathbf{Y}_{0..0} = \mathbf{Y}$
and assign to subsets $E \subseteq \mathbf{Y}$
the number
\[
\mu_\varphi(E)
=
\frac{1}{|\mathbf{Y}_{-N..N}|}
\sum_{
  \substack{
    x = (x_{-N} \rightarrow \dotsb \rightarrow x_{N}) \in
    \mathbf{Y}_{-N..N}:
    \\ 
    x_0 \in E
  }
}
|\varphi|^2(x),
\]
which quantifies
how much mass $\varphi : \mathbf{Y}_{-N..N} \rightarrow \mathbb{C}$
assigns to paths whose central vertex
lies in $E$.
\begin{question}\label{question:weak-limits}
  Fix $n \in \mathbb{Z}_{\geq 0}$.
  Let $N \geq n$ traverse a sequence of positive integers
  tending to $\infty$.
  For each $N$,
  choose an element $\varphi_N \in \mathcal{F}_{-N..N}$.
  What are the
  possible limits
  of the sequence of measures $\mu_{\varphi_N}$
  on the space $\mathbf{Y}_{-n..n}$?
\end{question}
The following conjecture
has not appeared explicitly
in the literature, but
may be regarded nowadays
as a standard
analogue
of
the arithmetic quantum unique ergodicity conjecture of
Rudnick--Sarnak \cite{MR1266075}
(cf. \cite{sarnak-progress-que, PDN-AP-AS-que} and references).
\begin{conjecture}\label{conj:duh-obvious-nowadays-AQUE}
  In the context of Question \ref{question:weak-limits},
  the uniform measure on $\mathbf{Y}_{-n..n}$ is the only
  possible weak limit.
  In other words,
  for any sequence $\varphi_N \in \mathcal{F}_{-N..N}$
  and any $E \subseteq \mathbf{Y}_{-n..n}$,
  \[
  \lim_{N \rightarrow \infty}
  \mu_{\varphi_N}(E)
  = \frac{|E|}{|\mathbf{Y}_{-n..n}|}.
  \]
\end{conjecture}
Conjecture \ref{conj:duh-obvious-nowadays-AQUE} predicts that
for any sequence $\varphi_N \in \mathcal{F}_{-N..N}$, the
corresponding sequence of $L^2$-masses $\mu_{\varphi_N}$
equidistributes under pushforward to any fixed space
$\mathbf{Y}_{-n..n}$.  One can formulate this conclusion more
concisely in terms of equidistribution on the compact space
$\varprojlim \mathbf{Y}_{-n..n}$ of infinite bidirectional
non-backtracking paths, or equivalently, on the space
$\mathbf{X} = \Gamma \backslash G$.

By explicating
the triple product formula \cite{MR2585578},
one can show that Conjecture
\ref{conj:duh-obvious-nowadays-AQUE}
follows from an open case of the subconvexity conjecture,
which in turn follows from GRH;
the latter can be shown
to imply more precisely that
\begin{equation}\label{eq:lindelof-prediction-booyah}
  \mu_{\varphi_N}(E)
  = \frac{|E|}{|\mathbf{Y}_{-n..n}|}
  + O(p^{-(1+o(1)) N/2})
\end{equation}
for fixed $n$.
There are nowadays
well-developed techniques
(see for instance \cite[\S1.4]{nelson-variance-73-2})
to show
that
\begin{itemize}
\item  the prediction \eqref{eq:lindelof-prediction-booyah}
holds
for $\varphi_N$ outside
a hypothetical exceptional subset of density $o(1)$,
that
\item
  if \eqref{eq:lindelof-prediction-booyah} is true,
  it is essentially optimal, and that
\item
  Conjecture \ref{conj:duh-obvious-nowadays-AQUE}
  holds for $\varphi_N$ outside
  a hypothetical exceptional subset of
  extremely small
  density
  $|\mathcal{F}_{-N..N}|^{-1/2+o(1)} = o(1)$,
\end{itemize}
but the problem of eliminating such exceptions
entirely (in the present setting and related ones)
has proved subtle.

For context, we recall some instances
in which the difficulty indicated above has been overcome;
notation and terminology should be clear by analogy.
\begin{theorem}[Lindenstrauss \cite{MR2195133}]\label{thm:lind}
  Let $\Delta \backslash \mathbb{H}$ be a compact hyperbolic
  surface attached to an order in  a non-split  indefinite quaternion
  algebra.  Let $\varphi$ traverse a sequence of $L^2$-normalized
  Hecke--Laplace eigenfunctions
  on
  $\Delta \backslash \mathbb{H}$
  with Laplace eigenvalue tending
  to $\infty$.  Then the $L^2$-masses
  $\mu_{\varphi}$ equidistribute.
\end{theorem}
\begin{theorem}[N, N--Pitale--Saha, Hu \cite{PDN-HQUE-LEVEL, PDN-AP-AS-que, 2014arXiv1409.8173H}]\label{thm:N-NPS}
  Fix a natural number $q_0$.  Let $q$ traverse a
  sequence of natural numbers tending to $\infty$.
  Let
  $\varphi$ be an $L^2$-normalized holomorphic Hecke newform
  on the standard congruence subgroup
  $\Gamma_0(q)$
  of $\SL_2(\mathbb{Z})$.  Then
  the pushforward to $\Gamma_0(q_0) \backslash \mathbb{H}$ of
  the $L^2$-mass of $\varphi$ equidistributes.
\end{theorem}
We may of course specialize Theorem \ref{thm:N-NPS} to powers of a fixed prime:
\begin{theorem}[N, N--Pitale--Saha, Hu \cite{PDN-HQUE-LEVEL, PDN-AP-AS-que, 2014arXiv1409.8173H}]\label{cor:N-NPS}
  Fix a prime $p$ and a nonnegative integer $n_0$.  Let $n$ traverse a
  sequence of natural numbers tending to $\infty$.
  Let
  $\varphi$ be an $L^2$-normalized holomorphic Hecke newform
  on
  $\Gamma_0(p^n)$.  Then
  the pushforward to $\Gamma_0(p^{n_0}) \backslash \mathbb{H}$ of
  the $L^2$-mass of $\varphi$ equidistributes.
\end{theorem}

Conjecture \ref{conj:duh-obvious-nowadays-AQUE}
is in the spirit of Theorem
\ref{cor:N-NPS}, save a crucial distinction to be discussed in
due course
(see Remark \ref{rmk:K0-case-easy}).
Unfortunately,
the method underlying the proof of Theorem
\ref{cor:N-NPS}, due to Holowinsky--Soundararajan \cite{MR2680499},
is fundamentally inapplicable to Conjecture
\ref{conj:duh-obvious-nowadays-AQUE}
due to its reliance on parabolic Fourier expansions,
which are unavailable on the compact quotient $\mathbf{X}$.
We will instead develop
here
a method more closely aligned with
that underlying the proof of Theorem \ref{thm:lind}.

To describe our result, we must recall that the elements of
$\mathcal{F}_{-N..N}$ may be partitioned according to the
isomorphism class of the representation of
$G = \GL_2(\mathbb{Q}_p)$ that they generate.  For $N \geq 1$,
any such representation is either
\begin{enumerate}
\item a (ramified) principal series representation (see
  \S\ref{sec:principal-series-reps}),
  or
\item a (supercuspidal) discrete series representation.
\end{enumerate}
A (computable) positive proportion of elements of
$\mathcal{F}_{-N..N}$ belongs to either category.  The dichotomy
here is analogous to that on
$\SL_2(\mathbb{Z}) \backslash \SL_2(\mathbb{R})$ between Maass forms
(principal series) and holomorphic forms (discrete series).
\begin{theorem}[Main result]\label{thm:super-simplified}
  The conclusion of Conjecture \ref{conj:duh-obvious-nowadays-AQUE}
  holds if $\varphi_N$ belongs to the principal series.
\end{theorem}
Theorem \ref{thm:super-simplified} represents the first genuine
instance of arithmetic quantum unique ergodicity in the level
aspect on a compact arithmetic quotient and also the first
on any $p$-adic arithmetic quotient.  It says that for
a sequence $\varphi_N \in \mathcal{F}_{-N..N}$ belonging to the
principal series, the corresponding $L^2$-masses equidistribute
under pushforward to any fixed space $\mathbf{Y}_{-n..n}$.  

\begin{remark}\label{rmk:fix-split-l-que-graph}
  Our
result might be described concisely as \emph{arithmetic quantum
  unique ergodicity on the path space over the fixed regular
  graph $(\mathbf{Y},T_p)$} and as contributing to the growing literature
concerning quantum chaos on regular graphs (see \cite{MR2677974,
  MR3038543, MR3322309} and references).
Alternatively, one could fix an auxiliary
split prime $\ell \neq p$, regard $(\mathbf{Y}_{-N..N}, T_\ell)$
as traversing an inverse system of $(\ell + 1)$-regular graphs,
and interpret Theorem \ref{thm:equid-balanced-newvectors} as a
form of arithmetic quantum unique ergodicity for such a
sequence of graphs.
\end{remark}

\begin{remark}
  Assuming the multiplicity
  hypothesis
  that
  an
  element $\varphi \in \mathcal{F}_{-N..N}$
  generating an irreducible principal series
  representation of $G$ is automatically
  an eigenfunction of the $T_\ell$ for $\ell \neq p$
  (which is inspired by analogy from the conjectural
  simplicity of the spectrum  of the Laplacian on
  $\SL_2(\mathbb{Z}) \backslash \mathbb{H}$),
  Theorem
  \ref{thm:equid-balanced-newvectors} may be understood as
  telling us something new about individual finite graphs
  $(\mathbf{Y},T_p)$, such as those pictured above, together with
  their realization as $\Gamma \backslash G / K$.
\end{remark}

As indicated already,
the proof of
Theorem \ref{thm:super-simplified}
is patterned on that of Theorem \ref{thm:lind}.
An important ingredient in the
proof of Theorem \ref{thm:lind} is the existence of a measure
$\mu$ on $\Delta \backslash \SL_2(\mathbb{R})$, called a
\emph{microlocal lift}, with the properties:
\begin{itemize}
\item $\mu$ lifts the measure
  $\lim_{j \rightarrow \infty} \mu_{\varphi_j}$ on
  $\Delta \backslash \mathbb{H}$.
\item $\mu$ is
  invariant under right translation by the diagonal subgroup of
  $\SL_2(\mathbb{R})$.
\item $(\mu_{\varphi_j})_j \mapsto \mu$
  is compatible with the Hecke operators
  (see \cite[Thm
  1.6]{MR2346281} for details);
  this third property
  is
  that which is not obviously satisfied
  by the classical construction via charts and pseudodifferential calculus.
\end{itemize}
The known construction of $\mu$ with such properties, due to
Zelditch and Wolpert (see \cite{MR916129, MR1814849,
  MR1859345})
and generalized by Silberman--Venkatesh \cite{MR2346281},
relies heavily upon explicit calculation with
raising and lowering operators in the Lie algebra of
$\SL_2(\mathbb{R})$, which have no obvious $p$-adic analogue.
One point of this paper is to introduce such an analogue and to
investigate systematically its relationship
to the classical theory of local newvectors.
 (The
restriction to principal series in Theorem
\ref{thm:super-simplified} then arises for the same reason that
Lindenstrauss's argument does not apply to holomorphic forms of
large weight: the absence of a ``microlocal lift'' invariant by
a split torus.)  The resulting construction may be of
independent interest; for instance, it should have applications
to the test vector problem (see \S\ref{sec:estim-l-funct} and
Remark \ref{remark:mv-epic}).

A curious subtlety of the argument, to be detailed further in
Remark \ref{rmk:lifting-behavior}, is that the ``lift'' we
construct is not a lift in the traditional sense (except against
spherical observables, and even then only for $p \neq 2$).  It
instead satisfies a weaker ``equidistribution implication''
property which suffices for us.  This subtlety is responsible
for the most technical component of the argument
(\S\ref{sec-stationary-phase-local-rs}).

In the remainder of \S\ref{sec:intro} we formulate our main
result in a slightly more general setup
(\S\ref{sec:main-results-general}), introduce a key
tool (\S\ref{sec:p-adic-microlocal}), give an overview of the
proof (\S\ref{sec:equid-micr-lifts}),
interpret our results in terms of $L$-functions
(\S\ref{sec:estim-l-funct}),
and record some further remarks and open questions
(\S\ref{sec:further-remarks}).

\subsection{Main results; general form}\label{sec:main-results-general}
In this section we formulate a generalization of Theorem
\ref{thm:lind} in representation-theoretic language, which
we adopt for the remainder of the paper.

\begin{definition}\label{defn:group-theoretic-paths}
Define
the compact open subgroup
\begin{equation}\label{eq:defn-K-m-mp}
  K_{m..m'}
  :=
  \begin{bmatrix}
    \mathfrak{o}  & \mathfrak{p}^{-m} \\
    \mathfrak{p}^{m'} & \mathfrak{o} 
  \end{bmatrix}^\times,
  \quad
  \mathfrak{o} := \mathbb{Z}_p, \mathfrak{p} := p \mathbb{Z}_p
\end{equation}
of $G$.
Each such subgroup is conjugate
to $K_{0..n}$ for $n = m' - m \geq 0$,
which is in turn analogous to the congruence
subgroup $\Gamma_0(p^n)$ of $\SL_2(\mathbb{Z})$.
Assuming \eqref{eq:torsion-free-assumption-yay},
one has compatible bijections
\[
\mathbf{X}/K_{m..m'}
= \Gamma \backslash G / K_{m..m'}
\xrightarrow{\cong }
\mathbf{Y}_{m..m'}
\]
\[
\Gamma g K_{m..m'}
\mapsto
(x_m \rightarrow x_{m+1} \rightarrow \dotsb \rightarrow x_{m'})
\text{ where }
x_j :=
\Gamma g
\begin{pmatrix}
  p^{-j} &  \\
  & 1
\end{pmatrix}
K
\]
with
$\mathbf{Y}_{m..m'}$
as defined in \eqref{eq:example-path-m-mp}.
\end{definition}
\begin{definition}\label{defn:eigenfunctions-measures-etc}
  The space $\mathcal{A}(\mathbf{X})$ of \emph{smooth} functions
  on $\mathbf{X}$ consists of all functions
  $\varphi : \mathbf{X} \rightarrow \mathbb{C}$ that are
  right-invariant under some open subgroup of $G$.  An
  \emph{eigenfunction} on $\mathbf{X}$ is an element
  $\varphi \in \mathcal{A}(\mathbf{X})$ that is a
  $T_\ell$-eigenfunction for each $\ell$ and that generates an
  irreducible representation of $G$ under the right translation
  action $g \varphi(x) := \rho_{\reg}(g) \varphi(x) := \varphi(x g)$.  The
  \emph{uniform measure} on $\mathbf{X}$, denoted simply
  $\int_{\mathbf{X}}$, is the probability Haar coming from the
  $G$-action.  An element $\varphi \in \mathcal{A}(\mathbf{X})$
  is \emph{$L^2$-normalized} if
  $\int_{\mathbf{X}} |\varphi|^2= 1$.  In that case, the
  \emph{$L^2$-mass} of $\varphi$ is the probability measure
  $\mu_\varphi$ on $\mathbf{X}$ given by
  $\mu_\varphi(\Psi) := \int_{\mathbf{X}} \Psi |\varphi|^2$.
  Convergence of measures always refers to the weak sense, i.e.,
  $\lim_{n \rightarrow \infty} \mu_n = \mu$ if for each fixed
  $\Psi \in \mathcal{A}(\mathbf{X})$,
  $\lim_{n \rightarrow \infty} \mu_n(\Psi) = \mu(\Psi)$.  A
  sequence of measures \emph{equidistributes} if it converges to
  the uniform measure.
\end{definition}

\begin{definition}\label{defn:spaces-aut-forms}
  We denote by
  $\mathcal{H} \subseteq \End(\mathcal{A}(\mathbf{X}))$ the ring
  generated by $\rho(G)$ and the $T_\ell$,
  so that an eigenfunction in the sense
  of Definition \ref{defn:eigenfunctions-measures-etc}
  is an element of $\mathcal{A}(\mathbf{X})$
  that generates an irreducible $\mathcal{H}$-submodule.
  We denote by $A(\mathbf{X})$
  the set of irreducible $\mathcal{H}$-submodules of
  $\mathcal{A}(\mathbf{X})$, by
  $A_0(\mathbf{X}) \subseteq A(\mathbf{X})$ the subset
  consisting of those that are not one-dimensional, and by
  $\mathcal{A}_0(\mathbf{X}) \subseteq \mathcal{A}(\mathbf{X})$
  the sum of the elements of $A_0(\mathbf{X})$,
  or equivalently, the orthogonal complement of the
  one-dimensional
  irreducible submodules.
\end{definition}  
A theorem of Eichler/Jacquet--Langlands implies that each
$\pi \in A(\mathbf{X})$ occurs in $\mathcal{A}(\mathbf{X})$ with
multiplicity one, so that
$\mathcal{A}(\mathbf{X}) = \oplus_{\pi \in A(\mathbf{X})} \pi$
and
$\mathcal{A}_0(\mathbf{X}) = \oplus_{\pi \in A_0(\mathbf{X})} \pi$.
The one-dimensional elements of $A(\mathbf{X})$ are given by
$\mathbb{C} (\chi \circ \det)$ for each character $\chi$ of the
compact group $\mathbb{Q}_p^\times/\det(\Gamma)$, thus
$A(\mathbf{X}) = \{\mathbb{C} (\chi \circ \det)\} \bigsqcup
A_0(\mathbf{X})$.
\begin{definition}\label{defn:central-char-conductor}
  Let $\chi_\pi : \mathbb{Q}_p^\times \rightarrow
  \mathbb{C}^\times$
  denote the central character of $\pi$.
  For $\pi \in A_0(\mathbf{X})$,
  the \emph{conductor} of $\pi$ has the form $C(\pi) = p^{c(\pi)}$,
  where $c(\pi)$
  is the
  smallest nonnegative integer
  with the property that
  $\pi$ contains a nonzero vector
  $\varphi$ satisfying
  $g \varphi = \chi_\pi(d) g$
  for all $g = \left(
    \begin{smallmatrix}
      \ast& \ast\\
      \ast & d
    \end{smallmatrix}
  \right) \in K_{0..c(\pi)}$ \cite{MR0337789, Sch02}.
\end{definition}
\begin{definition}\label{defn:newvectors-original}
  Let $\pi \in A_0(\mathbf{X})$.
  For integers $m,m'$,
  a vector $\varphi \in \pi$ will
  be called
  a \emph{newvector of support} $m..m'$
  if $m' - m = c(\pi)$
  and
  $g \varphi = \chi_\pi(d) \varphi$
  for all $g = \left(
    \begin{smallmatrix}
      \ast& \ast\\
      \ast & d
    \end{smallmatrix}
  \right) \in K_{m..m'}$.
  Local newvector theory \cite{MR0337789, Sch02} implies that the
  space of such vectors is one-dimensional,
  so if $\varphi$ is $L^2$-normalized,
  then the $L^2$-mass $\mu_\varphi$
  depends only upon $\pi$ and
  $m..m'$, not $\varphi$.
  A vector $\varphi \in \pi$
  will be called a \emph{generalized newvector}
  if it is a newvector of support $m..m'$ for some $m,m'$.
  (We include the adjective ``generalized''
    only to indicate explicitly
    that we are not necessarily referring
    to the traditional case $m..m'=0..c(\pi)$,
    which will play no distinguished role here.)
\end{definition}

\begin{remark}
  The newvectors of support $m..m'$ that generate
  representations with unramified (equivalently, trivial) central character
  may be characterized more simply
  as those eigenfunctions $\varphi \in \mathcal{A}(\mathbf{X})$
  (in the sense of Definition \ref{defn:eigenfunctions-measures-etc})
  which
  \begin{enumerate}
  \item are $K_{m..m'}$-invariant, or equivalently, descend to
    $\varphi : \mathbf{Y}_{m..m'} \rightarrow \mathbb{C}$, and
  \item are orthogonal to pullbacks from $\mathbf{Y}_{n..n'}$
    whenever $n..n' \subsetneq m..m'$.
  \end{enumerate}
  Under the torsion-freeness assumption
  \eqref{eq:torsion-free-assumption-yay},
  ``orthogonal'' can be taken
  to mean with respect
  to the  normalized counting measure
  on $\mathbf{Y}_{m..m'}$;
  in general,
  one should take that
  induced by the uniform measure on $\mathbf{X}$.
  In this sense,
  Definition \ref{defn:newvectors-original}
  is consistent with
  Definition \ref{defn:newvectors-super-classicla}.
\end{remark}

\begin{definition}
  We say that
$\pi \in A_0(\mathbf{X})$ \emph{belongs to the principal
  series} if the corresponding representation of $G$ does (see
\S\ref{sec:principal-series-reps}).
\end{definition}

\begin{theorem}[Equidistribution of newvectors, II]\label{thm:equid-balanced-newvectors}
  Let $\pi \in A_0(\mathbf{X})$ traverse  a sequence
  with $C(\overline{\pi} \times \pi) \rightarrow \infty$.
  Assume that $\pi$ belongs to the principal series.
  Let $\varphi \in \pi$ be an $L^2$-normalized generalized newvector.
  Then $\mu_\varphi$ equidistributes.
\end{theorem}
Theorem \ref{thm:equid-balanced-newvectors}
specializes to
Theorem \ref{thm:super-simplified}
upon requiring that $\chi_\pi$ be unramified (equivalently,
trivial)
and restricting
to newvectors of support $m..m' = -N..N$ for some $N$.

\begin{remark}
  Unlike earlier works such as \cite{PDN-HQUE-LEVEL,
    PDN-AP-AS-que, 2014arXiv1409.8173H}, we have allowed
  arbitrary central characters in Theorem
  \ref{thm:equid-balanced-newvectors}.  We
  note that the case of the argument in which the conductor of
  the central character is as large as possible relative to that
  of the representation is a bit more technically challenging than the
  others; see \eqref{eq:I3-defn-integral} and following.
\end{remark}

\begin{remark}\label{rmk:K0-case-easy}
  Cases of Theorem \ref{thm:equid-balanced-newvectors} in which
  $m..m'$ is highly unbalanced, such as the most traditional
  case $m..m' = 0..n$ analogous to Theorem \ref{cor:N-NPS},
  are easier: they follow, sometimes with a power savings, from
  the triple product formula, the convexity bound for triple
  product $L$-functions, and nontrivial local estimates as in
  \cite{PDN-AP-AS-que, 2014arXiv1409.8173H}.  Cases in which $m..m'$
  is balanced, such as the case $m..m' = -N..N$ illustrated
  in \S\ref{sec:overview}, do not follow from such local arguments and require the
  new ideas introduced here.  This phenomenon is comparable to
  how the mass equidistribution on a hyperbolic surface
  $\Delta \backslash \mathbb{H}$ of a weight $k$ vector in a
  principal series
  $\pi \hookrightarrow L^2(\Delta \backslash \SL_2(\mathbb{R}))$
  of parameter $t \rightarrow \infty$ follows from essentially
  local means for $t/k = o(1)$ but not for $k = 0$, or even for
  $k \ll t$; see \cite{MR1183602, MR1859598} for some discussion
  along such lines.
  See also Remark
  \ref{rmk:failure-converse-weakly-subconvex-newvectors-que}
  and footnote \ref{footnote:weak-sc}.
\end{remark}

\subsection{$p$-adic microlocal
  lifts}\label{sec:p-adic-microlocal}
We turn to the key definitions
that power the proof
of the above results.
We develop them slightly
more precisely and algebraically
than is strictly necessary for the consequences indicated above.

Let $k$ be a non-archimedean local field with ring of integers
$\mathfrak{o}$, maximal ideal $\mathfrak{p}$, normalized
valuation
$\nu : k \twoheadrightarrow \mathbb{Z} \cup \{+\infty\}$, and
$q := \# \mathfrak{o} / \mathfrak{p}$.  (The case
$(k,\mathfrak{o},\mathfrak{p},q) = (\mathbb{Q}_p,\mathbb{Z}_p, p
\mathbb{Z}_p, p)$ is relevant for the above application.)

To a
generic irreducible representation $\pi$ of $\GL_n(k)$ one may
attach a \emph{conductor} $C(\pi) = q^{c(\pi)}$, with
$c(\pi) \in \mathbb{Z}_{\geq 0}$.  One also defines
$c(\omega)$ for each character $\omega$ of $\mathfrak{o}^\times$;
it is the smallest integer $n$ for which $\omega$
has trivial restriction to $\mathfrak{o}^\times \cap 1 + \mathfrak{p}^n$.

For context,
we
record the local form of
Definition \ref{defn:newvectors-original}:
\begin{definition}[Newvectors]
  \label{defn:newvectors-general}
  A vector $v$ in an irreducible generic representation
  $\pi$ of $\GL_2(k)$
  is a \emph{newvector of support $m..m'$}
  if $m' - m = c(\pi)$
  and
  \[
    \pi(g) v
    = \chi_\pi(d) v
    \text{ for all }
    g = \begin{pmatrix}
      a & b \\
      c & d
    \end{pmatrix}
    \in \GL_2(\mathfrak{o})
    \cap \begin{bmatrix}
      \mathfrak{o}  & \mathfrak{p}^{-m} \\
      \mathfrak{p}^{m'} & \mathfrak{o} 
    \end{bmatrix}.
  \]
  A \emph{generalized newvector} is a newvector of some support.
\end{definition}

Fix now for each nonnegative
integer $N$ a partition $N = N_1 + N_2$ into nonnegative integers $N_1,N_2$
with the property that $N_1, N_2 \rightarrow \infty$ as $N
\rightarrow \infty$.
The precise choice is unimportant;
one might take $N_1 := \lfloor N/2 \rfloor, N_2 := \lceil N/2
\rceil$
for concreteness.
Using this choice,
we introduce the following class of vectors:
\begin{definition}[Microlocal lifts]\label{defn:microlocal-lifts}
  Let $\pi$ be a $\GL_2(k)$-module.  A vector $v \in \pi$ shall
  be called a \emph{microlocal lift} if it generates an
  generic irreducible admissible representation of $\GL_2(k)$
  and if
  there is a positive integer $N$ and characters
  $\omega_1, \omega_2$ of $\mathfrak{o}^\times$
  so that
  $c(\omega_1/\omega_2) = N$
  and
  \[
    \pi(g)
    v = \omega_1(a) \omega_2(\det(g)/a) v
    \text{ for all }
    g = \begin{pmatrix}
      a & b \\
      c & d
    \end{pmatrix}
    \in \GL_2(\mathfrak{o})
    \cap \begin{bmatrix}
      \mathfrak{o}  & \mathfrak{p}^{N_1} \\
      \mathfrak{p}^{N_2} & \mathfrak{o} 
    \end{bmatrix}.
  \]
  In that
  case, we refer to $N$ as the \emph{level} and
  $(\omega_1,\omega_2)$ as the \emph{orientation} of
  $v$.
\end{definition}

The observation
that
the special case $\omega_1 = 1$ of Definition
\ref{defn:microlocal-lifts}
is similar to Definition \ref{defn:newvectors-general}
leads easily to the following characterization of
microlocal lifts as twists of generalized
newvectors from ``extremal principal series'' representations
``$1 \boxplus \chi$'' (see
\S\ref{sec:proof-determination-microlocal} for the proof):
\begin{lemma}\label{thm:compute-microlocal-lifts}
  An irreducible admissible representation $\pi$ of $\GL_2(k)$
  contains a nonzero microlocal lift if only if
  $\pi$ is an irreducible principal
  series representation $\pi \cong \chi_1 \boxplus \chi_2$
  for which
  $N := c(\overline{\pi}
  \otimes \pi)/2  = c(\chi_1/\chi_2)$ is nonzero.
  In that case,
  the set of microlocal lifts is the union of two distinct
  lines.
  One line consists of microlocal lifts of level $N$ and
  orientation $(\omega_1,\omega_2)$,
  where
  $\omega_i := \chi_i|_{\mathfrak{o}^\times}$;
  explicitly, it is the inverse image
  under the non-equivariant twisting isomorphism
  $\pi \rightarrow \pi \otimes \chi_1^{-1} \cong 1 \boxplus \chi_1^{-1} \chi_2$
  of the space of newvectors of support $-N_1..N_2$.
  The other line is described similarly
  with the roles of $\omega_1$ and $\omega_2$ reversed.
\end{lemma}

\begin{remark}
  We briefly compare
  with the archimedean analogue inspiring Definition \ref{defn:microlocal-lifts};
  a more complete exposition of this analogy
  seems beyond the scope of this article.
  Let $\pi$ be a principal series representation of
  $\PGL_2(\mathbb{R})$ of parameter $t \rightarrow \pm \infty$
  with lowest weight vector $\varphi_0$ corresponding to a
  spherical Maass form of eigenvalue $1/4 + t^2$ on some
  hyperbolic surface.  The Zelditch--Wolpert construction\footnote{
    We discuss here only the ``positive measure'' incarnation of that construction rather than the ``distributional'' one.
  } 
  of a microlocal lift $\varphi_1$ of $\varphi_0$ is given up to normalizing factors in
  terms of standard raising/lowering operators $X^n$ for
  $n \in \mathbb{Z}$ (see \cite{MR1814849, MR1859345})  by
  $\varphi_1 := \sum_{n:|n| \leq t_1} X^n \varphi_0$, where
  $|t| = t_1 t_2$ with $t_1, t_2 \rightarrow \infty$ as
  $|t| \rightarrow \infty$.  The choice
  $\varphi_2 := \sum_{n:|n| \leq t_1} (-1)^n X^n \varphi_0$
  also works.
  The analogue
  of
  $(|t|, \varphi_1, \varphi_2, |.|^{i t}, |.|^{-i t})$
  in
  the notation of
  Definition \ref{defn:microlocal-lifts}
  and Lemma \ref{thm:compute-microlocal-lifts}
  is
  $(q^N, v_1, v_2, \chi_1, \chi_2)$ with $q := \#
  \mathfrak{o}/\mathfrak{p}$
  and $v_1,v_2 \in \pi$ microlocal lifts
  of respective orientations $(\omega_1,\omega_2)$,
  $(\omega_2,\omega_1)$.
  The analogy
  may be obtained by comparing
  how $\GL_2(\mathfrak{o})$
  acts on $v_1,v_2$ to how
  the Lie algebra of $\PGL_2(\mathbb{R})$ acts on $\varphi_1,\varphi_2$.
  The factorization $|t| = t_1 t_2$ is roughly
  analogous to the partition $N = N_1 + N_2$.
  It is also instructive
  to compare the formulas for $\varphi_1,\varphi_2$
  in their induced models with those of \S\ref{sec-explicit-formulas}.
\end{remark}
\begin{remark}\label{rmk:other-microlocal}
  Le-Masson \cite{MR3245884} and Anantharaman--Le-Masson
  \cite{MR3322309} have introduced a notion of microlocal lifts
  on regular graphs and used that notion to prove some analogues
  of the quantum ergodicity theorem.  Definition
  \ref{defn:microlocal-lifts} serves different aims in
  that we do not explicitly vary the
  graph (except perhaps in the second sense indicated in Remark
  \ref{rmk:fix-split-l-que-graph});
  it would be interesting to extend it further and
  compare the two notions on any domain of overlap.
\end{remark}
    

For the remainder of \S\ref{sec:p-adic-microlocal},
take $k = \mathbb{Q}_p$,
so that $\GL_2(k) = G$.
Definition \ref{defn:microlocal-lifts}
applies to $\pi \in A_0(\mathbf{X})$.
\begin{theorem}[Basic properties of microlocal lifts]\label{thm:microlocal-basics}
  Let $N$ traverse a sequence of positive integers
  tending to $\infty$,
  and let $\varphi \in \pi \in A_0(\mathbf{X})$ be
  an $L^2$-normalized microlocal lift of level $N$ on
  $\mathbf{X}$
  with $L^2$-mass $\mu_\varphi$.
  \begin{itemize}
  \item {\bf Diagonal invariance.}
    
    Any weak subsequential limit of the sequence of measures $\mu_\varphi$
    is $a(\mathbb{Q}_p^\times)$-invariant.
  \item
    { \bf Lifting property.}
    
    Suppose temporarily that $p \neq 2$, so that $\nu(2) =
    0$.
    Let
    $\varphi' \in \pi$ be an $L^2$-normalized newvector
    of support $-N..N$,
    and let $\Psi \in \mathcal{A}(\mathbf{X})^K$
    be independent of $N$ and right invariant by
    $K := \GL_2(\mathbb{Z}_p)$.
    Then \[\lim_{N \rightarrow \infty} (\mu_{\varphi}(\Psi) -
    \mu_{\varphi'}(\Psi)) = 0.\]
  \item
    {\bf Equidistribution implication.}

    Suppose that $\mu_\varphi$ equidistributes as $N
    \rightarrow \infty$.
    Let $\varphi' \in \pi$
    be an $L^2$-normalized generalized newvector.
    Then $\mu_{\varphi'}$ equidistributes as $N \rightarrow \infty$.
  \end{itemize}
\end{theorem}

Theorem \ref{thm:microlocal-basics} is established in
\S\ref{sec:microlocal-basics} after developing the necessary
local preliminaries in \S\ref{sec-33} and
\S\ref{sec:basic-theory-microlocal}.  The proof involves
uniqueness of invariant trilinear forms\footnote{ It should be
  possible to avoid this comparatively deep fact in the proof of the first part of Theorem
  \ref{thm:microlocal-basics}, but it is required by the
  application to subconvexity (Theorem
  \ref{thm:weakly-subconvex}), and the calculations required by
  that application already suffice here.
}
on $\GL_2$
and stationary phase analysis of local Rankin--Selberg
integrals.
Theorem \ref{thm:microlocal-basics} is essentially local, i.e.,
does not exploit the arithmeticity of $\Gamma \leq G$,
and is stated here in a global setting only for convenience;
see Theorem \ref{thm:local-rs} for a local analogue.

\begin{remark}\label{rmk:lifting-behavior}
  The ``lifting property'' of Theorem
  \ref{thm:microlocal-basics} has been included only for the
  sake of illustration; it is not strictly necessary for the
  logical purposes of this paper.  We have assumed $p \neq 2$ in
  its statement because the corresponding assertion is false
  when $p=2$.  For general $p$ and non-spherical observables
  $\Psi$, there does not appear to be any simple relationship
  between the quantities $\mu_{\varphi}(\Psi)$ and
  $\mu_{\varphi'}(\Psi)$
  except that convergence to $\int_{\mathbf{X}} \Psi$
  of the first implies that of the
  second (the ``equidistribution implication'').
  The ``lifting'' relationship here
  is thus more subtle than that in \cite{MR2195133}.
\end{remark}

\subsection{Equidistribution of microlocal lifts}\label{sec:equid-micr-lifts}
Our core result (from which the others are ultimately derived)
is the following:
\begin{theorem}[Equidistribution of microlocal lifts]\label{thm:main}
  Let $N$ traverse a sequence of positive integers tending to
  $\infty$.
  Let $\varphi \in \mathcal{A}(\mathbf{X})$ be an $L^2$-normalized microlocal lift of level $N$ on $\mathbf{X}$.
  Then $\mu_\varphi$ equidistributes.
\end{theorem}
The proof depends upon an analogue
of Lindenstrauss's celebrated result \cite{MR2195133}:
\begin{theorem}[Measure classification]\label{thm:measure-classification}
  Let $\mu$ be a probability measure on $\mathbf{X}$,
  invariant by the center of $G$, with the properties:
  \begin{enumerate}
  \item $\mu$ is $a(\mathbb{Q}_p^\times)$-invariant.
  \item $\mu$ is $T_\ell$-recurrent for some split prime $\ell \neq
    p$.
  \item The entropy of almost every ergodic component of $\mu$ is
    positive
    for the $a(\mathbb{Q}_p^\times)$-action.
  \end{enumerate}
  Then $\mu$ is the uniform measure.
\end{theorem}
We explain in \S\ref{sec:measure-classification}
the specialization of Theorem \ref{thm:measure-classification}
from a result of Einsiedler--Lindenstrauss \cite[Thm 1.5]{MR2366231}.
To deduce Theorem \ref{thm:main}, we apply Theorem \ref{thm:measure-classification} 
with $\mu$ any weak limit of the $L^2$-masses of a sequence of
$L^2$-normalized microlocal lifts of level tending to $\infty$.  Since $\mathbf{X}$ is compact, $\mu$ is a probability measure.
The invariance hypothesis follows from the diagonal invariance of Theorem \ref{thm:microlocal-basics},
while the $T_\ell$-recurrence and positive entropy hypotheses
are verified below in \S\ref{sec:recurrence} and \S\ref{sec:pos-ent}.
The proof of our main result Theorem \ref{thm:main} is then complete.
Theorem \ref{thm:main}
and the equidistribution implication of Theorem \ref{thm:microlocal-basics}
imply Theorem \ref{thm:equid-balanced-newvectors}.

\subsection{Estimates for
  $L$-functions}\label{sec:estim-l-funct}
For definitions of the $L$-functions and local distinguishedness see \cite{MR911357,
  MR2449948}.
We record the following because it provides an unambiguous
benchmark
of the strength of our results.

\begin{theorem}[Weakly subconvex bound]\label{thm:weakly-subconvex}
  Fix $\sigma \in A_0(\mathbf{X})$.  Let
  $\pi \in A_0(\mathbf{X})$ traverse a sequence with
  $C(\overline{\pi} \times \pi) \rightarrow \infty$.  Assume
  that $\pi$ belongs to the principal series and that
  $\sigma \otimes \overline{\pi} \otimes \pi$ is locally
  distinguished.  Then
  \begin{equation}\label{eq:weakly-subconvex}
    \frac{L(\sigma \times \overline{\pi} \times \pi, 1/2)}{
      L(\ad \pi, 1)^2
    }
    = o(C(\sigma \times \overline{\pi} \times \pi)^{1/4}).
  \end{equation}
\end{theorem}

The previously best known estimate for the LHS of
\eqref{eq:weakly-subconvex} is the general weakly subconvex
estimate of Soundararajan \cite{soundararajan-2008},
specializing here to
$L \ll C^{1/4} / (\log C)^{1-\eps}$
with $L := L(\sigma \times \overline{\pi} \times \pi, 1/2)$,
$C := C(\sigma \times \overline{\pi} \times \pi)$.
The bound
\eqref{eq:weakly-subconvex} improves upon that estimate in the
unlikely (but difficult to exclude) case that $L(\ad \pi,1)$ is
exceptionally small, which turns out to be the most difficult
one for equidistribution problems; see \cite{MR2680499} for
further discussion.

Theorem \ref{thm:main} implies Theorem
\ref{thm:weakly-subconvex} after a local calculation with the
triple product formula (see \S\ref{sec:microlocal-basics}); in
fact, the calculation shows that the two results are equivalent.

\begin{remark}\label{rmk:failure-converse-weakly-subconvex-newvectors-que}
  Theorem \ref{thm:weakly-subconvex}
  implies Theorem \ref{thm:equid-balanced-newvectors}, but
  the converse does not hold in general; a special case of the
  failure of that converse was noted and discussed at length in
  \cite[\S1]{PDN-AP-AS-que}.  The present work may thus be
  understood as clarifying that discussion: the equivalence between subconvexity
  and equidistribution problems in the depth aspect is restored by
  working not with newvectors, but instead with the $p$-adic
  microlocal lifts introduced here.
\end{remark}

\subsection{Further remarks}\label{sec:further-remarks}

\begin{remark}
  Theorems \ref{thm:equid-balanced-newvectors} and
  \ref{thm:main}
  apply only to sequences of
  vectors $\varphi$ that
  generates irreducible $\mathcal{H}$-modules.
  One can ask whether the
  conclusion holds under the (hypothetically) weaker assumption
  that $\varphi$ generates
  an irreducible $G$-module.  The problem formulated this way
  makes sense for any finite volume quotient
  $\Gamma' \backslash G$, not necessarily 
  arithmetic; an affirmative answer would represent a $p$-adic
  analogue of the Rudnick--Sarnak quantum unique ergodicity
  conjecture \cite{MR1266075}.  In that direction, we note that
  the method of Brooks--Lindenstrauss \cite{MR3260861} should
  apply in our setting, allowing one to relax the hypothesis of
  irreducibility under the full Hecke algebra to that under a
  single auxiliary Hecke operator $T_\ell$ for some fixed split
  prime $\ell \neq p$.
\end{remark}


\begin{remark}
  Our results apply to principal series representations of
  conductor $p^N$ with $p$ fixed and $N \rightarrow \infty$.
  A natural question is whether one can establish analogous
  results for $N$ fixed, such as $N = 100$, and
  $p \rightarrow \infty$.  We highlight here the weaker question
  of whether one can establish equidistribution (in a balanced case, cf. Remark \ref{rmk:K0-case-easy}) as
  $N \rightarrow \infty$ for $p$ satisfying $p \leq p_0(N)$ for
  some $p_0(N)$ tending \emph{effectively} to $\infty$ as
  $N \rightarrow \infty$.  Our results and a diagonalization
  argument imply an ineffective analogue.
\end{remark}

\begin{remark}
  The crucial local results of this article have been formulated
  and proved in generality, i.e., over any non-archimedean local field.
  On the other hand,
  we have assumed in our global results that the subgroup
  $\Gamma$ of $G$ was constructed from a \emph{maximal order}
  in a quaternion algebra \emph{over $\mathbb{Q}$}.  We expect that our
  results hold more generally:
  \begin{enumerate}
  \item The statements and proofs of all our results except
    Theorem \ref{thm:weakly-subconvex}
    extend straightforwardly to the case that
    $\Gamma$ arises from a fixed \emph{Eichler order} in a
    quaternion algebra over $\mathbb{Q}$.  To extend Theorem
    \ref{thm:weakly-subconvex}
    in that direction would require some local triple
    product estimates at the ``uninteresting'' primes $\ell \neq p$ which we
    do not pursue here.
  \item 
    Our results should extend to Eichler
    orders in totally definite quaternion algebras over
    \emph{totally
      real number fields}, but some mild care is required in formulating such
    extensions when the class group has nontrivial $2$-torsion: as observed in a related context in
    \cite{PDN-HMQUE}, there are sequences of dihedral forms that
    fail to satisfy the most naive formulation of quantum unique
    ergodicity.
  \item We expect our results extend to automorphic forms
    on definite quaternion algebras having
    fixed nontrivial infinity type; such an extension would
    require a more careful study of the measure classification
    input
    in \S\ref{sec:measure-classification}.
  \item Over function fields,
    analogues of our results should follow more directly and in
    quantitatively stronger forms from Deligne's theorem and
    extensions of the triple product formula to the function
    field setting.
  \end{enumerate}
  We leave such extensions to the interested reader.
\end{remark}



\subsection*{Organization of this paper.}
We verify the measure-classification (Theorem \ref{thm:measure-classification}) and its hypotheses
in \S\ref{sec:measure-classification}, \S\ref{sec:recurrence},
\S\ref{sec:pos-ent}.
We review the representation theory of $\GL_2(k)$
in \S\ref{sec-33}.
In \S\ref{sec:basic-theory-microlocal} and \S\ref{sec:microlocal-basics},
we prove our core results,
notably Theorem
\ref{thm:microlocal-basics}, and their applications.
Some additional results of independent interest are
recorded
along the way.


\subsection*{Acknowledgements}
This paper owes an evident debt of
ideas and inspiration to E. Lindenstrauss's work
\cite{MR2195133};
we thank him also for helpful feedback and interest.  We thank M. Einsiedler for helpful
discussions on measure classification and feedback on an earlier draft, P. Sarnak and
A. Venkatesh for several helpful discussions informing our
general understanding of microlocal lifts and microlocal
analysis, A. Saha for helpful references concerning conductors,
S. Jana for feedback on entropy bounds,
Y. Hu
for helpful clarifying questions,
and E. Kowalski, Ph.
Michel and D. Ramakrishnan for encouragement.
We gratefully acknowledge the support of
NSF grant OISE-1064866
and SNF grant SNF-137488 during the work leading
to this paper.

\section{Measure classification\label{sec:measure-classification}}
\label{sec-3}
The purpose of this section is to deduce Theorem \ref{thm:measure-classification} 
from the following specialization to $\mathbb{Q}_p$
of a result of Einsiedler--Lindenstrauss \cite[Thm 1.5]{MR2366231}:
\begin{theorem}\label{thm:EL}
  Let $G = G_1 \times G_2$, where $G_1$ is a semisimple linear
  algebraic group over $\mathbb{Q}_p$ with $\mathbb{Q}_p$-rank $1$ and
  $G_2$ is a characteristic zero $S$-algebraic group.  Let
  $\Delta \subset G$ be a discrete subgroup.  Let $A_1$ be a
  $\mathbb{Q}_p$-split torus of $G_1$ and let $\chi$ be a nontrivial
  $\mathbb{Q}_p$-character of $A_1$ that can be extended to $C_{G_1}(A_1)$.
  Let $M_1 = \{h \in C_{G_1}(A_1) : \chi(h) = 1\}$.  Let $\nu$ be an
  $A_1$-invariant, $G_2$-recurrent probability measure on
  $\Delta \backslash G$ such that
  \begin{enumerate}
  \item almost every $A_1$-ergodic component of $\nu$ has
    positive ergodic theoretic entropy with respect to some
    $a \in A_1$ with $|\chi(a)| \neq 1$, and
  \item 
    for $\nu$-a.e. $x \in \Delta \backslash G$, the group
    $\{h \in M_1 \times G_2 : x h = x\}$ is finite.
  \end{enumerate}
  Then $\nu$ is a convex combination of
  homogeneous measures,
  each of which is supported on an orbit of a subgroup $H$
  which contains a finite index subgroup
  of a semisimple algebraic subgroup
  of $G_1$ of $\mathbb{Q}_p$-rank one.
\end{theorem}

To deduce Theorem \ref{thm:measure-classification}
from Theorem \ref{thm:EL} requires no new ideas,
but we record a complete verification for completeness.

\subsection{Consequences of strong approximation}
\label{sec:strong-approx}
Recall that $R$ is a maximal order
in a definite quaternion algebra $B$.
For a prime $p$,
we shall use the notations
$B_p := B \otimes_{\mathbb{Q}} \mathbb{Q}_p$,
$R_p := R \otimes_{\mathbb{Z}} \mathbb{Z}_p$.
A superscripted $(1)$ denotes ``norm one elements,''
e.g.,  $B_p^{(1)} := \{b \in B_p^\times : \nr(b) = 1\}$.
Denote by $\mathbb{A}_f$ the finite adele ring of $\mathbb{Q}$
and $\hat{B} := B \otimes_{\mathbb{Q}} \mathbb{A}_f$.
Regard $B^\times, B_p^\times, R_p^\times$ as subsets of $\hat{B}^\times$ in the standard way.
\begin{lemma}\label{lem:strong-approx-crit}
  Let $H$ be a subgroup of $\hat{B}^\times$
  for which:
  \begin{enumerate}
  \item[(i)]
    There is a prime $p$ that splits $B$
    for which $H$ contains an open subgroup
    of $\hat{B}^{(1)}$ containing $B_p^{(1)}$.
  \item[(ii)]
    The image $\nr(H)$ of $H$ under
    the reduced norm $\nr :  \hat{B}^\times \rightarrow
    \mathbb{A}_f^\times$
    satisfies $\mathbb{Q}_+^\times \nr(H) = \mathbb{A}_f^\times$.
  \end{enumerate}
  Then $B^\times H = \hat{B}^\times$.
\end{lemma}
\begin{proof}
  It is known
  that $\nr : B^\times \rightarrow \mathbb{Q}_+^\times$ is surjective.
  Let $b \in \hat{B}^\times$
  be given.
  By (ii),
  there exists $\gamma \in B^\times$
  and $h \in H$
  for which $\gamma b h \in \hat{B}^{(1)}$.
  Let $p$ be as in (i).
  Strong approximation for the simply-connected semisimple
  algebraic group $B^{(1)}$
  and its non-compact factor $B_p^{(1)}$
  implies that $B^{(1)} B_p^{(1)}$ is dense in $\hat{B}^{(1)}$.
  By (i), we may write
  $\gamma b h = \delta h'$ for some $\delta \in B^{(1)}$ and $h'
  \in H$.
  Therefore $b = \gamma^{-1} \delta h' h^{-1}$ belongs to
  $B^\times H$,
  as required.
\end{proof}
Let $p$ be a split prime for $B$.
For any prime $\ell$,
one has $\nr(B_\ell^\times) = \mathbb{Q}_\ell^\times$;
because $R$ is a maximal order (in particular, an Eichler order),
one has moreover that $\nr(R_\ell^\times) = \mathbb{Z}_\ell^\times$.
The hypotheses of Lemma \ref{lem:strong-approx-crit}
thus apply to
$H = B_p^\times \prod_{\ell \neq p} R_\ell^\times$:
(i) is clearly satisfied,
while (ii) follows
from the consequence
$\mathbb{Q}_+^\times
\mathbb{Q}_p^\times \prod_{\ell \neq p} \mathbb{Z}_\ell^\times =
\mathbb{A}_f^\times$
of strong approximation for the ideles.
For similar but simpler reasons, the hypotheses
apply also to $H = B_p^\times B_\ell^\times \prod_{q \neq \ell,p}
R_q^\times$.
Thus
\[
B^\times B_p^\times \prod_{\ell \neq p} R_\ell^\times
= \hat{B}^\times
=
B^\times B_p^\times B_\ell^\times  \prod_{q \neq \ell, p} R_q^\times.
\]
We have
$B^\times \cap \prod_{\ell \neq p} R_\ell^\times =
R[1/p]^\times$
and
$B^\times \cap \prod_{q \neq \ell, p} R_q^\times =
R[1/p \ell]^\times$,
whence the natural identifications
\begin{equation}\label{eq:strong-approx-consequences-prelim}
  R[1/p]^\times \backslash B_p^\times/\mathbb{Q}_p^\times 
=
B^\times \backslash \hat{B}^\times / \mathbb{Q}_p^\times \prod_{\ell \neq p}
R_\ell^\times
=
R[1/p \ell]^\times \backslash B_p^\times B_\ell^\times / \mathbb{Q}_p^\times R_\ell^\times.
\end{equation}
Since $\mathbb{Z}[1/p \ell]^\times \mathbb{Q}_p^\times
\mathbb{Z}_\ell^\times = \mathbb{Q}_p^\times
\mathbb{Q}_\ell^\times$,
the RHS of \eqref{eq:strong-approx-consequences-prelim}
is unaffected by further reduction modulo
$\mathbb{Q}_\ell^\times$,
i.e.,
\begin{equation}\label{eq:strong-approx-consequences}
R[1/p]^\times \backslash B_p^\times/\mathbb{Q}_p^\times
=
R[1/p \ell]^\times \backslash B_p^\times B_\ell^\times / \mathbb{Q}_p^\times \mathbb{Q}_\ell^\times R_\ell^\times.
\end{equation}
\subsection{Deduction of Theorem
  \ref{thm:measure-classification}}
Let $p$ be a split prime for $B$.
Identify $B_p^\times  = \GL_2(\mathbb{Q}_p)$
and
$\mathbf{X} = \Gamma \backslash \GL_2(\mathbb{Q}_p)$
as in \S\ref{sec:intro}.
Let $\mu$ be a measure 
on $\mathbf{X}$
satisfying
the hypotheses of Theorem \ref{thm:measure-classification}.
It is invariant under the diagonal torus of $\GL_2(\mathbb{Q}_p)$,
which generates the latter modulo $\SL_2(\mathbb{Q}_p)$,
so to prove that $\mu$ is the uniform measure,
we need only verify that it is $\SL_2(\mathbb{Q}_p)$-invariant.
To that end, we apply Theorem \ref{thm:EL}:
Set
$G_1 := \PGL_2(\mathbb{Q}_p) = B_p^\times / \mathbb{Q}_p^\times
$, $G_2 := \PGL_2(\mathbb{Q}_\ell)
= B_\ell^\times / \mathbb{Q}_\ell^\times$, $G := G_1 \times G_2$.
Recall that
$\Gamma = R[1/p]^\times$.
Take for $\Delta$ the image of $R[1/p \ell]^\times$ in $G$.
By strong approximation in the form \eqref{eq:strong-approx-consequences},
we may identify
$\Gamma \backslash \GL_2(\mathbb{Q}_p) / \mathbb{Q}_p^\times$
with
$\Delta \backslash G / \PGL_2(\mathbb{Z}_\ell)$
and $\mu$ with a right $\PGL_2(\mathbb{Z}_\ell)$-invariant
measure $\nu$ on $\Delta \backslash G$.  
Our task is then to verify that $\nu$ is invariant by the image of $\SL_2(\mathbb{Q}_p)$.
Take for $A_1$ the diagonal torus in $G_1$
and for $\chi : A_1 \rightarrow \mathbb{Q}_p^\times$
the map $\chi(\diag(y_1,y_2)) := y_1/y_2$.
We have $C_{G_1}(A_1) = A_1$.
The group $M_1$ is trivial, hence
each $\{h \in M_1 \times G_2 : x h = x\}$ is trivial.  
The hypotheses
of Theorem \ref{thm:measure-classification} are satisfied,
so $\nu$ is invariant by some finite index subgroup $H_1$
of some semisimple algebraic subgroup of $G_1$ (of $\mathbb{Q}_p$-rank one) that
contains $A_1$.  The smallest such $H_1$ is the image of $\SL_2(\mathbb{Q}_p)$,
so we conclude.

\section{Recurrence\label{sec:recurrence}}
\label{sec-4}
In this section we formulate and verify the $T_\ell$-recurrence
hypothesis required by Theorem \ref{thm:measure-classification}.
The argument here is as in
\cite[\S8]{MR2195133} except that we allow general central
characters; for completeness, we record a proof of the key
estimate in that case.  The proof is simple; a key insight of
Lindenstrauss \cite{MR2195133} is that the condition enunciated
here is useful for the present purposes.

\subsection{Hecke operators}\label{sec:definition-t_n}
For a positive integer $n$ coprime to $p$, the Hecke operator
$T_n \in \End(\mathcal{A}(\mathbf{X}))$ is defined by
$T_n \varphi(x) := \sum_{\alpha \in M_n/M_1} \varphi(\alpha^{-1}
x)$,
where $M_n := R[1/p] \cap \nr^{-1}(n \mathbb{Z}[1/p]^\times)$,
so that $M_1 = \Gamma$.  These operators commute with one
another and also with $\rho_{\reg}(G)$.
If $\ell \mid \disc B$, then
the operator $T_\ell$ is an involution modulo the action of the
center; otherwise, it is induced by a correspondence of degree
$\ell+1$.
For $m \in \mathbb{Q}^\times$, abbreviate
$z(m) := \rho_{\reg}(z(m))$
(see \S\ref{sec:p-adic-microlocal} for notation).
The adjoint of $T_n$ is
$T_n^* = z(n) T_n$, and one has the composition formula
\begin{equation}\label{eq:hecke-multiplicativity}
  T_m T_n = \sum_{\substack{
      d \in \mathbb{Z}_{\geq 1} : \\
      d \mid \gcd (m,n),  \,     \gcd(d,\disc B)=1 
    }
  } d \cdot z(d^{-1}) T_{m n/d^2}.
\end{equation}

\subsection{Spherical averaging operators}\label{sec:spher-aver-oper}
Let $n$ be a positive integer coprime to $p$.
The operator $T_n$ on
$\mathcal{A}(\mathbf{X})$
is induced by the \emph{correspondence} on $\mathbf{X}$,
denoted also by
$T_n$,
given for $x \in \mathbf{X}$
by the multiset (i.e., formal sum) $T_n(x) := \sum_{s \in M_n/\Gamma} s^{-1}
x$.
Thus $T_n \varphi(x) = \sum_{y \in T_n(x)} \varphi(y)$.
Denote by $M_n^{\prim}$ the set of all \emph{primitive} elements
of $M_n$, i.e., those that are not divisible inside $R[1/p]$ by any divisor $d >
1$ of $n$.
Then $M_n^{\prim}$ is right-invariant by $\Gamma$,
and one has
$M_n = \bigsqcup_{d^2 \mid n}
  z(d) M_{n/d^2}^{\prim}$.
Denote by $S_n$ the ``Hecke sphere'' correspondence
$S_n (x) := \sum_{s \in M_n^{\prim}/\Gamma} s^{-1} x$;
it likewise induces an operator $S_n$
on $\mathcal{A}(\mathbf{X})$
given by
$S_n \varphi(x) := \sum_{s \in M_{n}^{\prim}/\Gamma}
\varphi(s^{-1} x) = \sum_{y \in S_n(x)} \varphi(y)$,
and one has
\begin{equation}\label{eq:decomposition-M-n-into-primitives}
  T_n(x) = \sum_{d^2 \mid n} z(d^{-1}) S_{n/d^2}(x).
\end{equation}

\subsection{Recurrence}
Let $\ell \neq p$ be a \emph{split prime}, that is to say,
a prime that splits the quaternion algebra underlying
the construction of $\Gamma$, so that the Hecke operator
$T_\ell$ has degree $\ell + 1$.
\begin{definition}
  A finite $Z$-invariant measure $\mu$ on $\mathbf{X}$
  is called \emph{$T_\ell$-recurrent}
  if for each Borel subset $E \subseteq \mathbf{X}$
  and $\mu$-almost every $x \in E$,
  there exist infinitely many positive integers $n$ 
  for which $S_{\ell^n}(x) \cap E \neq \emptyset$.
\end{definition}

\begin{theorem}[Hecke recurrence]\label{thm:recurrence}
  Let $\mu$ be any subsequential limit of a sequence of
  $L^2$-masses $\mu_\varphi$ of $L^2$-normalized
  automorphic forms $\varphi \in \pi  \in A_0(\mathbf{X})$.
  Then $\mu$ is $T_\ell$-recurrent.\footnote{It suffices
    to assume only that $\varphi$ is a $T_\ell$-eigenfunction.}
\end{theorem}
The proof of Theorem
\ref{thm:recurrence}
reduces via measure-theoretic considerations as in 
\cite{MR2195133, MR3260861}
to that of the following:

\begin{lemma}
  There exists $c_0 > 0$
  so that for each split prime $\ell$
  and $\varphi \in \pi \in A_0(\mathbf{X})$ and $x \in
  \mathbf{X}$,
  one has
  $\sum_{k \leq n} \sum_{y \in S_{\ell^k}(x)} |\varphi(y)|^2 \geq c_0 n |\varphi(x)|^2$.
\end{lemma}

\begin{proof}
  By a theorem of Eichler, Shimura and Igusa,
  $\pi$ is tempered,\footnote{
    As in the references, the non-tempered case may be treated
    more simply.}
  hence there exist $\alpha, \beta \in \mathbb{C}^{(1)}$ (the
  Satake parameters)
  so that $\lambda_\pi(\ell) = \alpha + \beta$;
  one then has more generally for $n \in \mathbb{Z}_{\geq 1}$
  that
  \begin{equation}\label{eq:hecke-eigenvalue-formula}
    \lambda_\pi(\ell^n) = \frac{\alpha^{n+1} - \beta^{n+1}}{\alpha - \beta }.
  \end{equation}
  By \eqref{eq:decomposition-M-n-into-primitives}, one has
  $T_{\ell^n} = \sum_{k \leq n : k \equiv n(2)}
  z(\ell^{(k-n)/2}) S_{\ell^k}$.
  Conversely,
  $S_{\ell^k} = T_{\ell^k} - 1_{k \geq 2} z(\ell^{-1}) T_{\ell^{k-2}}$.
  Since $\pi$ has a unitary central character,
  there is $\theta \in \mathbb{C}^{(1)}$
  so that $z(\ell^{-1}) \varphi = \theta \varphi$
  for all $\varphi \in \pi$.
  Thus, denoting by $\ell^{k/2} \sigma_k \in \mathbb{C}$
  the scalar by which $S_{\ell^k}$ acts on $\pi$,
  one obtains
  $\sigma_k = \lambda(\ell^k) - 1_{k \geq 2}
  \theta \ell^{-1} \lambda(\ell^{k-2})$,
  which expands for $k \geq 2$ to
  \begin{equation}\label{eq:sigma-k-via-satake-parameters}
    \sigma_k
    =
    \frac{\gamma_1 \alpha^{k} - \gamma_2 \beta^{k}}{
      \alpha - \beta }.
  \end{equation}
  with $\gamma_1 := \alpha - \theta \ell^{-1} \alpha^{-1},
  \gamma_2 :=  \beta - \theta \ell^{-1} \beta^{-1}$.
  Note that $|\gamma_1|, |\gamma_2| \geq 1/2$.

  We turn to the main argument.
  For $m, k \in \mathbb{Z}_{\geq 0}$,
  Cauchy--Schwarz gives
  \begin{align*}
    \ell^m |\lambda_\pi(\ell^{m}) \varphi(x)|^2
    &= | T_{\ell^{m}} \varphi(x)|^2
      \leq (1 + \ell^{-1}) \ell^{m}
      \sum_{y \in T_{\ell^m}(x)} |\varphi(y)|^2,
    \\
    \ell^k|\sigma_{k} \varphi(x)|^2
    &= |S_{\ell^k} \varphi(x)|^2
      \leq (1 + \ell^{-1})\ell^{k}
      \sum_{y \in S_{\ell^k}(x)}
      |\varphi(y)|^2,
  \end{align*}
  whence by \eqref{eq:decomposition-M-n-into-primitives}
  that
  $\sum_{k \leq n} \sum_{y \in S_{\ell^k}(x)} |\varphi(y)|^2
  \gg
  |\varphi(x)|^2
  c_{\pi,\ell}(n)$
  with
  \begin{equation}\label{eq:defn-c-pi-n}
    c_{\pi,\ell}(n)
    :=
    \sum_{k \leq n}
    |\sigma_{k}|^2 
    + \max_{m \leq n} |\lambda_\pi(\ell^{m})|^2.
  \end{equation}
  Our task thereby reduces to verifying that
  $c_{\pi,\ell}(n) \gg n$, uniformly in $\pi$ and
  (unimportantly) $\ell$.  Suppose this estimate fails.
  Then there is a sequence of integers $j \rightarrow \infty$
  and tuples $(\pi,n,\ell) = (\pi_j,n_j,\ell_j)$ as above,
  depending upon $j$, so that
  $n \rightarrow \infty$ as $j \rightarrow \infty$ and
  $c_{\pi,\ell}(n) = o(n)$.  Here asymptotic notation refers to
  the $j \rightarrow \infty$ limit, and for quantities
  $A,B = A_j, B_j$ depending (implicitly) upon $j$, we write
  $A \ll B$ for
  $\limsup_{j \rightarrow \infty} |A_j/B_j| < \infty$ and
  $A \lll B$ or $A = o(B)$ for
  $\limsup_{j \rightarrow \infty} |A_j/B_j| = 0$; the notations
  $A \gg B$ and $A \ggg B$ are defined symmetrically.  We shall
  derive from this supposition a contradiction.  By passing to
  subsequences, we may consider separately cases in which the
  Satake parameters $\alpha,\beta$ of $\pi$, as defined above,
  satisfy (i) $|\alpha - \beta| \ggg 1/n$ or (ii)
  $|\alpha - \beta| \ll 1/n$.\footnote{ The standard argument
    considers cases for which
    $|\alpha - \beta| \gg 1/n$ and $|\alpha - \beta| \lll 1/n$.
    We have found the present division slightly more efficient.
  } In case (i), we have
  $|1 - \alpha \overline{\beta}|^{-1} \lll n$, and so upon
  expanding the square and summing the geometric series,
  \[
  c_{\pi,\ell}(n) \geq 
  \sum_{k \leq n} |\sigma_{k}|^2
  =
  \frac{|\gamma_1|^2 n + |\gamma_2|^2 n + o(n)}{|\alpha - \beta|^2}
  \geq
  \frac{n/3}{|\alpha - \beta|^2}
  \gg n.
  \]
  In case (ii),
  one has
  $|\alpha - \beta|^{-1}/10 \gg n$,
  so the largest positive integer $m \leq n$
  for which $m |\alpha - \beta| < 1/10$
  satisfies $m \gg n$, and \eqref{eq:hecke-eigenvalue-formula}
  gives
  $c_{\pi,\ell}(n) \geq |\lambda_\pi(\ell^{m})|^2 \gg m^2 \gg n^2 \geq n$.
  In either case, we derive the required contradiction.
\end{proof}

\section{Positive entropy\label{sec:pos-ent}}
\label{sec-5}
In this section we verify the entropy hypothesis required by
Theorem \ref{thm:measure-classification}.  The basic ideas here
are due to Bourgain--Lindenstrauss \cite{MR1957735} following
earlier work of Rudnick--Sarnak \cite{MR1266075} and Lindenstrauss \cite{MR1859345}
and followed by later developments of Silberman--Venkatesh \cite{SV-AQUE}  and
Brooks--Lindenstrauss \cite{MR3260861}.  Those works dealt with archimedean
aspects;
the present $p$-adic adaptation
is 
obtained by replacing the role played by the discreteness
of $\mathbb{Z}$ in $\mathbb{R}$ with that of
$\mathbb{Z}[1/p]$ in $\mathbb{R} \times \mathbb{Q}_p$.
We also give a
new formulation of the basic line of attack (Lemma
\ref{lem:geom-ampl}) emphasizing convolution over covering
arguments
(compare with \cite[Lemma 3.4]{SV-AQUE}),
which may be of use in other contexts.

Call $\eps > 0$ \emph{admissible} if it belongs to the image of $|.| :
\mathbb{Q}_p^\times \rightarrow \mathbb{R}_+^\times$.
For a compact open subgroup $C$ of $\mathbb{Q}_p^\times$
and admissible $\eps > 0$ 
set
\[
B(C,\eps) := \{\begin{pmatrix}
  a  & b \\
  c & d
\end{pmatrix} \in G : a,d \in C, |b|, |c| \leq \eps\}.
\]
We refer to \cite[\S8]{MR2459293} for definitions and basic
facts concerning measure-theoretic entropy.  As in
\cite[\S8]{MR2459293} or \cite{MR1957735} or \cite[Thm 6.4]{SV-AQUE}, the
following criterion suffices:
\begin{theorem}[Positive entropy on almost every ergodic component]\label{thm:pos-ent}
  For each compact subset $\Omega$ of $G$ there exists
  $C$ as above and $C_1, c_2 > 0$ so that for all admissible $\eps \in (0,1)$, all
  $L^2$-normalized $\varphi \in \pi \in  A_0(\mathbf{X})$,
  and all $x \in \Omega$, one has
  $\mu_\varphi(x B(C,\eps)) \leq C_1 \eps^{c_2}$.
\end{theorem}

Given $\Omega$, we choose for $C$ any open subgroup of $\mathfrak{o}^\times$
with the property that
for small enough $\eps$,
one has
\begin{equation}\label{eq:small-enough-in-K}
  \text{$x B(C,\eps) x^{-1} \subseteq K$
    for all $x \in \Omega$},
\end{equation}
\begin{equation}\label{eq:small-enough-wrt-units}
  \text{$g B(C,\eps) g^{-1} \cap \Gamma = \{1\}$ for all $g \in G$}
\end{equation}
the latter being possible because $B$ is non-split.
We now state two independent lemmas,
prove Theorem \ref{thm:pos-ent}
assuming them, and then prove the lemmas.
\begin{lemma}[Bounds for Hecke returns]\label{lem:hecke-returns}
  For all admissible
  $\eps \in (0,1)$, all $n \in \mathbb{Z}_{\geq 1}$ coprime to $p$ and satisfying
  $n < \sqrt{1/2} \eps^{-1}$,
  all $m \in \mathbb{Q}^\times$ with numerator and denominator
  coprime to $p$,
  and all $x \in \Omega$,
  the set $S := M_n \cap z(m) x B(C,\eps) x^{-1}$ has cardinality $\# S \leq 6 \prod_{p^k || n} (k + 1)$.  In particular, $\# S \leq 2^{13}$ if $n$ has at most $10$ prime divisors counted with multiplicity.
\end{lemma}
\begin{lemma}[Geometric amplification]\label{lem:geom-ampl}
  Let $(c_\ell)_{\ell \in \mathbb{Z}_{\geq 1}}$
  be a finitely-supported sequence of scalars.
  Set
  $T := \sum_{\ell}
  c_{\ell} T_\ell / \sqrt{\ell}$
  and
  $T^a := \sum_{\ell}
  |c_\ell| T_\ell^* / \sqrt{\ell}$.
  Let $\varphi \in \mathcal{A}(\mathbf{X})$,
  $\psi, \nu \in C_c^\infty(G)$.
  Define $\Psi \in \mathcal{A}(\mathbf{X})$ by
  $\Psi(g) := \sum_{\gamma \in \Gamma} |\psi|(\gamma g)$
  and $\psi \ast \nu \in C_c^\infty(G)$
  by $\psi \ast \nu(x) := \int_{y \in G}
  \psi(x y) \nu(y)$.
  Then
  \[
  \|T \varphi (\psi \ast \nu) \|_{L^2(G)}
  \leq
  \|\varphi\|_{L^2(\mathbf{X})}
  \|T^a \Psi\|_{L^2(\mathbf{X})}
  \|\nu\|_{L^2(G)}.
  \]
\end{lemma}
\begin{proof}[Proof of Theorem \ref{thm:pos-ent}]
  We have $T \varphi = \lambda \varphi$ with $\lambda := \sum
  c_\ell \lambda_\pi(\ell)$ where $T_\ell \varphi = \sqrt{\ell} \lambda_\pi(\ell) \varphi$.
  Abbreviate
  $J := B(C,\eps)$; it is a group.
  Let $x \in \Omega$.
  Take $\psi := 1_{x B(C,\eps)} \geq 0$
  and $\nu := e_J := \vol(J)^{-1} 1_J$.
  Then $1_{x B(C,\eps)} = |\psi \ast \nu|^2$.
  By (\ref{eq:small-enough-wrt-units}),
  we have
  $\mu_{T \varphi}(|\psi \ast \nu|^2) = \|T \varphi (\psi \ast \nu) \|_{L^2(G)}^2$,
  and so by Lemma \ref{lem:geom-ampl},
  $\mu_{\varphi}(x B(C,\eps))
  \leq
  |\lambda|^{-1} \|T^a \Psi \|_{L^2(\mathbf{X})} \|\nu \|_{L^2(G)}$.
  The
  square $\|T^a \Psi \|_{L^2(\mathbf{X})}^2$
  is a linear combination of terms
  $\langle T_\ell^* \Psi, T_{\ell'}^* \Psi  \rangle = \langle
  T_{\ell'} T_\ell^* \Psi,  \Psi  \rangle$
  to which we apply
  the Hecke multiplicativity \eqref{eq:hecke-multiplicativity}
  and the unfolding: for $m,n \in \mathbb{Z}_{\geq 1}$,
  \begin{align}\label{eq:entropy-unfolding-hecke}
    \langle z(m) T_n^* \Psi, \Psi  \rangle \|\nu \|^2_{L^2(G)}
    &= \int_{g \in G}
    \sum_{s \in M_n}
      \psi(z(m) s g) \psi(g)\vol(J)^{-1}
      \\ \nonumber
    &= \# M_n \cap z(m^{-1}) x J x^{-1}.
  \end{align}
  By Lemma \ref{lem:hecke-returns}, we thereby obtain
  \[
  \mu_\varphi(x B(C,\eps))^2
  \leq 2^{13}
  |\lambda|^{-2} 
  \sum_{\ell, \ell'}
  |c_{\ell} c_{\ell'}|
  \sum_{d \mid (\ell,\ell')}
  d / \sqrt{\ell \ell '}
  \]
  provided that $c_\ell$ is supported on integers $\ell \leq 2^{-1/4} \eps^{-1/2}$
  having at most $5$ prime factors counted with multiplicity.
  A standard choice of $c_\ell$ completes the proof.
  For completeness, we record a variant of the choice from \cite[\S4.1]{venkatesh-2005}:
    Set $L := (1/\eps)^{0.1}$.
    Denote by $\mathcal{L}$ the set consisting of all $\ell = q$
    or $\ell = q^2$ taken over primes $q \in [L, 2 L]$; each such $q$ splits $B$ provided $\eps$ is small enough.
    Set $c_\ell := 0$ unless $\ell \in
    \mathcal{L}$,
    in which case
    $c_\ell := L^{-1} \log(L) \sgn(\lambda_\pi(\ell))^{-1}$.
    We have $\sum_{\ell} |c_\ell| \asymp 1$ and
    $|c_\ell| \leq L^{-1} \log(L)$,
    while Iwaniec's trick
    $|\lambda_\pi(q)|^2 + |\lambda_\pi(q^2)| \geq 1$, a
    consequence of
    (\ref{eq:hecke-multiplicativity}),
    implies
    $\lambda \asymp 1$.  With
    trivial estimation we obtain
    $\mu_\varphi(x B(C,\eps)) \ll L^{-1/2} (\log L)^{O(1)} \ll \eps^{0.01}$, as required.
\end{proof}
\begin{proof}[Proof of Lemma \ref{lem:hecke-returns}]
  Observe first, thanks to \eqref{eq:small-enough-in-K} and $n\mathbb{Z}[1/p]^\times \cap (\mathbb{Q}_+ \times \mathbb{Z}_p) = \{n\}$ and $z(m) \in K$,
  that $S \subseteq M_n \cap K = R(n) := \{\alpha \in R : \nr(\alpha) = n\}$.
  Given $s, t \in S$,
  their commutator $u := s t s^{-1} t^{-1}$
  thus satisfies
  $\nr(u) = 1$
  and
  $n^2 u  = s t s^{\iota} t^{\iota} \in R$, hence $\tr(u) \in n^{-2} \mathbb{Z}$.
  Since $S$ is
  conjugate to a subset of the preimage in $M_2(\mathfrak{o})$
  of the upper-triangular Borel in
  $M_2(\mathfrak{o}/\mathfrak{q})$
  with $\mathfrak{q} := \{x \in \mathfrak{o} : |x| \leq
  \eps^2\}$,
  and the commutator of that preimage is contained in the preimage of the unipotent,
  one has $|\tr(u) - 2|_p \leq \eps^2$.
  Since $B$ is definite,
  $|\tr(u)|_\infty \leq 2 |\nr(u)|_\infty^{1/2} = 2$.
  The integer $a := n^2 \tr(u) - 2 n^2$ thus satisfies
  $|a|_\infty |a|_p \leq 2 n^2 \eps^2 < 1$ and so must be zero,
  i.e., $\tr(u) = 2$;
  since $B$ is non-split, $u=1$.
  In summary, any two elements of $S$ commute.
  Since $B$ is non-split and definite,
  $S$ is contained in the set
  $\mathcal{O}(n)$ of norm $n$ elements
  in some imaginary quadratic
  order $\mathcal{O} \subset R$.
  Thus $\#S \leq \# \mathcal{O}(n) \leq \# \mathcal{O}^\times \cdot \#
  \{I \subseteq \mathcal{O} : \nr(I) = n\}
  \leq 6 \prod_{p^k || n} (k + 1)$.
\end{proof}
\begin{proof}[Proof of Lemma \ref{lem:geom-ampl}]
  Write $M := R[1/p]$.  We may express the operator $T$ by the
  formula
  $T \varphi(x) = \sum_{s \in M/\Gamma} h_s \varphi(s^{-1} x)$ for
  some finitely supported coefficients $h_s$; then
  $T^a \Psi(x) = \sum_{s \in M/\Gamma} |h_s| \Psi(s x)$.
  Abbreviate
  $I := \|T \varphi (\psi \ast \nu)\|_{L^2(G)}$.
  By the triangle inequality and a change of variables
  $x \mapsto s x$,
  we have
  \[I
  \leq
  \sum_{s \in M / \Gamma}
  |h_s|
  (
  \int_{x \in G}
  |\varphi|^2(x) |\psi \ast \nu(s x)|^2
  )^{1/2}.
  \]
  By a change of variables,
  $\psi \ast \nu(s x)
  = \int_{y \in G}
  \psi(s y)
  \nu_y^*(x)$
  with
  $\nu_y^*(x) := \nu(x^{-1} y)$.
  By the triangle inequality,
  $I
  \leq
  \int_{y \in G}
  \sum_{s \in M / \Gamma}
  |h_s|
  |\psi(s y)|
  \|\varphi \nu_y^*\|_{L^2(G)}$.
  We unfold
  $\int_{y \in G} \sum_{s \in M / \Gamma}
  = \int_{y \in \mathbf{X}} \sum_{s \in \Gamma \backslash M} \sum_{\gamma \in \Gamma}$,
  giving
  $I
  \leq
  \int_{y \in \mathbf{X}}
  T^a \Psi(y)
  \|\varphi \nu_y^*\|_{L^2(G)}$.
  We conclude via Cauchy--Schwartz
  and the identity
  $\int_{y \in \mathbf{X}}
  \|\varphi \nu_y^*\|_{L^2(G)}^2
  = \|\nu \|_{L^2(G)}^2
  \|\varphi\|_{L^2(\mathbf{X})}^2$.
\end{proof}

\section{Representation-theoretic preliminaries}
\label{sec-33}
\subsection{Generalities}
\label{sec-33-1}
Let $k$ be a non-archimedean local field
with maximal order $\mathfrak{o}$,
maximal ideal $\mathfrak{p}$,
normalized valuation $\nu : k \rightarrow \mathbb{Z} \cup
\{+\infty \}$,
and $q := \# \mathfrak{o} / \mathfrak{p}$.
Fix Haar measures $d x, d^\times y$ on $k, k^\times$ assigning
volume one to maximal compact subgroups.
Fix a nontrivial unramified additive
character $\psi : k \rightarrow \mathbb{C}^{(1)}$.
Set $G := \GL_2(k)$.

\subsection{Some notation and terminology}
\label{sec:rep-theory-prelims-some-notation}
For $x \in k$ and $y_1,y_2 \in  k^\times$, set
\[
n(x) := \begin{pmatrix}
  1 & x \\
  0 & 1
\end{pmatrix},
\quad
n'(x) :=
\begin{pmatrix}
  1 & 0 \\
  x & 1
\end{pmatrix},
\]
\[
\diag(y_1,y_2) := \begin{pmatrix}
  y_1 & 0 \\
  0 & y_2
\end{pmatrix},
\quad
w  := \begin{pmatrix}
  & -1 \\
  1 & 
\end{pmatrix}
\]
and
$a(y) := \diag(y,1)$,
$z(y) := \diag(y,y)$.
Say that a vector $v$  in some $\GL_2(k)$-module $\pi$
is \emph{supported on $m..m'$},
for integers $m, m'$ with $m \leq m'$,
if $v$ is invariant by $n(\mathfrak{p}^{-m})$ and
$n'(\mathfrak{p}^{m'})$,
and that $v$ has \emph{orientation $(\omega_1,\omega_2)$},
for characters $\omega_1, \omega_2$ of $\mathfrak{o}^\times$,
if $\pi(\diag(y_1,y_2)) v = \omega_1(y_1) \omega_2(y_2) v$ for
all $y_1,y_2 \in \mathfrak{o}^\times$.

\subsection{Principal series representations}\label{sec:principal-series-reps}
\label{sec-33-2}
For characters $\chi_1, \chi_2 : k^\times \rightarrow
\mathbb{C}^\times$,
denote by $\pi = \chi_1 \boxplus \chi_2$ the \emph{principal
  series representation}
of $G$ realized in its \emph{induced model}
as a space of smooth functions
$v : G \rightarrow \mathbb{C}$
satisfying $v(n(x) \diag(y_1,y_2) g) =
|y_1/y_2|^{1/2} \chi_1(y_1) \chi_2(y_2) v(g)$
for all $x \in k$ and $y_1,y_2 \in k^\times$ and $g \in G$.
A sufficient condition for $\pi$ to be irreducible
is that $c(\chi_1/\chi_2) \neq 0$.
If $\chi_1, \chi_2$ are unitary, then $\pi$ is unitary;
an invariant norm is given by
$\|v\|^2 := \int_{x \in k} |v(n'(x))|^2 \, d x$.
The log-conductor is $c(\pi) = c(\chi_1) + c(\chi_2)$
and the central character is $\chi_\pi = \chi_1 \chi_2$.

The following ``line model'' parametrization of $\pi$ shall be convenient:
for suitable $f \in C^\infty(k)$,
define $v_f \in \pi$
by
\begin{equation}\label{eqn:v-from-f-induced-model}
  v_f(g) :=
  f (c/d)
  \left\lvert \det (g)/d^2 \right\rvert^{1/2}
  \chi_1(\det (g)/d) \chi_2(d),
  \quad
  g = \begin{pmatrix}
    a & b \\
    c & d
  \end{pmatrix}. 
\end{equation}
If $\chi_1, \chi_2$ are unitary,
then $\|v_f\|^2 = \int_{k} |f|^2$.

\subsection{Generic representations}
\label{sec-33-3}
Recall that an irreducible representation
$\sigma$ of $G$ is \emph{generic}
if is isomorphic
to an irreducible subspace
$\mathcal{W}(\sigma,\psi)$
of the space of smooth functions $W : G \rightarrow \mathbb{C}$
satisfying $W(n(x) g) = \psi(x) W(g)$ for all $x,g \in k,G$;
in that case, $\mathcal{W}(\sigma,\psi)$ is called the
\emph{Whittaker model}
of $\sigma$.
It is known that every non-generic irreducible representation
of $G$ is one-dimensional.

For each $W \in \mathcal{W}(\sigma,\psi)$,
denote also by $W$ the function
$W : k^\times \rightarrow \mathbb{C}$
defined by $W(y) := W(a(y))$.
The space $\mathcal{K}(\sigma,\psi)$ of
functions $W : k^\times \rightarrow \mathbb{C}$
arising in this way from some $W \in \mathcal{W}(\sigma,\psi)$
is called the \emph{Kirillov model} of $\sigma$.
It is known that $\mathcal{K}(\sigma,\psi) \supseteq
C_c^\infty(k^\times)$.

An irreducible principal series representations $\pi = \chi_1
\boxplus \chi_2$
is generic;
the standard intertwining map from $\pi$ to its $\psi$-Whittaker
model $\mathcal{W}(\pi,\psi)$,
denoted $\pi \ni v \mapsto W_v : \GL_2(k) \rightarrow \mathbb{C}$,
is given by
$W_v(g) := \int_{x \in k}
v(w n(x) g) \psi(- x) \, d x$.
In general, this integral fails to converge absolutely and must
instead be interpreted via analytic continuation,
regularization, or as a limit of integrals taken over the compact
subgroups $\mathfrak{p}^{-n}$ of $k$ as $n \rightarrow \infty$
(see e.g. \cite[p485]{MR1431508}); we ignore such
standard technicalities here.

\subsection{Newvector theory}
\label{sec-33-4}
Recall Definition \ref{defn:newvectors-general}.
\begin{theorem}[Basic newvector theory]\label{thm:basic-local-newvectors}
  Let $\pi$ be a generic irreducible representation
  of $\GL_2(k)$
  and let $m \leq m'$ be integers.
  Then the space of vectors in $\pi$ supported on
  $m..m'$ and with orientation
  $(1,\chi_\pi|_{\mathfrak{o}^\times})$
  has dimension $\max(0,1 + |m-m'| - c(\pi))$.

  In particular,
  let $\pi$ be any irreducible representation of
  $\GL_2(k)$
  with ramified central character $\chi_\pi$.
  Denote by $V$ the space of vectors in $\pi$ supported on $-N_1..N_2$
  with orientation
  $(1,\chi_\pi|_{\mathfrak{o}^\times})$.
  Then $V = 0$ unless $\pi$ is generic,
  in which case $\dim V = \max(0,1 + c(\pi) - N)$.
\end{theorem}
\begin{proof}
  For the first assertion, see \cite{MR0337789}.
  The generic case of the second assertion
  follows from the first assertion,
  so suppose $\pi$ is one-dimensional.
  Write $\pi = \chi \circ \det$ for some $\chi : k^\times
  \rightarrow \mathbb{C}^\times$.
  Since $\chi_\pi$ is ramified, the characters $(1, \chi_\pi|_{\mathfrak{o}^\times})$
  and $(\chi|_{\mathfrak{o}^\times},
  \chi|_{\mathfrak{o}^\times})$
  of $\mathfrak{o}^\times \times
  \mathfrak{o}^\times$ are distinct, and 
  so $V = 0$.
\end{proof}

\begin{lemma}\label{lem:extremal-central-chars}
  Let $\pi$ be an irreducible generic representation of
  $\GL_2(k)$ with ramified central character
  $\chi_\pi$.
  Then $c(\chi_\pi) \leq c(\pi)$ with equality
  precisely when
  $\pi$ is isomorphic to an irreducible principal
  series representation $\chi_1 \boxplus \chi_2$
  for which at least one of the inducing characters $\chi_1,\chi_2$ is unramified.
\end{lemma}
\begin{proof}
  This is well-known, see \cite[Lemma 3.1]{MR3272013} or
  \cite[Proof of Prop 2]{MR0338274}.
\end{proof}

\begin{lemma}\label{lem:balanced-newvectors-in-induced-model}
  Let $\pi = \chi_1 \boxplus \chi_2$ be an irreducible principal
  series representation of $G$.  Let $v \in \pi$ be a newvector
  of some support $m..m'$.
  \begin{enumerate}
  \item If $\chi_1$ is ramified and
    $\chi_2$ is ramified, then $v = v_f$ as in
    \eqref{eqn:v-from-f-induced-model} for $f$ a
    character multiple of the characteristic function
    of an $\mathfrak{o}^\times$-coset,
    thus
    $f = c \chi 1_{\varpi^n \mathfrak{o}^\times}$
    for some $c \in \mathbb{C}$, $\chi : k^\times \rightarrow \mathbb{C}^\times$
    and $n \in \mathbb{Z}$.
  \item If $\chi_1$ is unramified and
    $\chi_2$ is ramified,
    then $v = v_f$
    for $f = c 1_\mathfrak{a}$
    for some scalar $c$ and fractional $\mathfrak{o}$-ideal $\mathfrak{a} \subset k$.
  \end{enumerate}
\end{lemma}
\begin{proof}
  Both assertions are well-known in the special case
  $m = 0$ (see \cite{Sch02}) and follow inductively in general
  using that $a(\varpi)$
  bijectively maps newvectors of support $m..m'$
  to those of support $m-1..m'-1$.
\end{proof}

\subsection{Local Rankin--Selberg integrals}
\label{sec-33-5}
Let $\pi$ be an irreducible unitary principal series
representation of $G := \GL_2(k)$
and $\sigma$ an irreducible generic unitary representation of
$\PGL_2(k)$.
We have the following special case of a theorem of D. Prasad:

\begin{theorem}\cite{MR1059954}\label{thm:uniqueness-trilinear-functionals}
  The space $\Hom_G(\sigma \otimes \overline{\pi} \otimes \pi, \mathbb{C})$,
  consisting of trilinear functionals $\ell : \sigma \otimes
  \overline{\pi } \otimes \pi \rightarrow \mathbb{C}$
  satisfying the diagonal invariance
  $\ell(\sigma(g) v_1, \overline{\pi}(g) \overline{v_2}, \pi(g) v_3)
  =
  \ell(v_1,v_2,v_3)$ for all $g \in G$ and all vectors,
  is one-dimensional.
\end{theorem}
We may fix a nonzero element $\ell_{\RS} \in \Hom_G(\sigma
\otimes \overline{\pi} \otimes \pi,\mathbb{C})$
as follows:
Denote by $Z$ the center of $G$
and
$U := \{n(x) : x \in k\}$.
Equip the right $G$-space $Z U \backslash G$
with the Haar measure for which
\begin{equation}\label{eq:haar-rs-defn}
  \int_{g \in Z U \backslash G}
  \phi(g)
  =
  \int_{y \in k^\times}
  \int_{x \in k}
  \phi(a(y) n'(x))
  \, \frac{d^\times y}{|y|} \, d x
\end{equation}
for $\phi \in C_c( Z U \backslash G)$ (see
\cite[\S3.1.5]{michel-2009}).
Realize $\pi$ in its induced model.
For $W_1 \in \mathcal{W}(\sigma,\psi),
W_2 \in \mathcal{W}(\pi,\psi)$
and $v_3 \in \pi$,
set 
$\ell_{\RS}(W_1, \overline{W_2}, v_3)
:=
\int_{Z U \backslash G }
W_1 \overline{W_2} v_3$
(see \cite[\S3.4.1]{michel-2009}).  The definition applies in particular when $W_2$ is the image $W_v$ of  some $v \in \pi$ under the intertwiner from \S\ref{sec-33-3}.

The trick encapsulated by the following lemma
(a careful application of ``non-archimedean integration by parts'')
shall be exploited repeatedly in \S\ref{sec-stationary-phase-local-rs}:
\begin{lemma}[Application of diagonal invariance]\label{lem:appl-diag-invar}
  Let $f \in C_c^\infty(k)$.  Let $U_1$ be an open subgroup of $\mathfrak{o}^\times$
  for which $\overline{f} \otimes f$ is \emph{$U_1$-invariant}
  in the sense that $\overline{f}(u x) f(u y) = \overline{f}(x)
  f(y)$
  for all $u,x,y \in U_1,k,k$.
  Let $W_1 \in \mathcal{W}(\sigma,\psi)$.
  Then
  \[
  \ell_{\RS}(W_1,\overline{W_{v_f}}, v_f) =
  \int_{\substack{x \in k  \\ y \in k^\times \\ t \in k} }
  f(x)
  \overline{f} \left( x + \frac{y}{t} \right)
  F(x,y,t;W_1,U_1)
  \, \frac{d t}{|t|}
  \, d x \, d^\times y,
  \]
  where
  $F(x,y,t;W_1,U_1) :=
  \mathbb{E}_{u \in U_1}
  W_1(a(y) n'(x/u))
  \chi_1 \chi_2^{-1}(u t)
  \psi(u t)$
  with $\mathbb{E}_{u \in U_1}$ denoting an integral with respect to the probability Haar.
\end{lemma}
\begin{proof}
  Set
  $g := a(y) n'(x) = \begin{pmatrix}
    y &  \\
    x & 1
  \end{pmatrix}$.
  For $t \in k$ one has
  $w n(t) g
  = 
  \begin{pmatrix}
    - x & -1 \\
    y + t x & t
  \end{pmatrix}$,
  hence
  \begin{align*}
    v_f(g)
    &=
      f(x)
      |y|^{1/2}
      \chi_1(y), \\
    \overline{v_f}(w n(t) g)
    &=
      \overline{f} \left( (y + t x)/t \right)
      \left\lvert y/t^2 \right\rvert^{1/2}
      \chi_1^{-1}(y/t) \chi_2^{-1}(t),
    \\
    v_f(g)
    \overline{W_{v_f}}(g)
    &=
      \int_{t \in k}
      v_f(g)
      \overline{v_f(w n(t) g)
      \psi(-t) }
      \, d t
    \\
    &=
      |y|
      f(x)
      \int_{t \in k}
      \overline{f} \left( x + \frac{y}{t} \right)
      \chi_1 \chi_2^{-1}(t)
      \psi(t)
      \, \frac{d t}{|t|}.
  \end{align*}
  Integrating 
  against $W_1(a(y) n'(x)) |y|^{-1} \, d x \, d^\times y$
  gives
  that
  $\ell_{\RS}(W_1, \overline{W_{v_f}}, v_f)$ equals
  \[
  \int_{\substack{x \in k  \\ y \in k^\times \\ t \in k} }
  f(x)
  \overline{f} \left( x + \frac{y}{t} \right)
  W_1(a(y) n'(x))
  \chi_1 \chi_2^{-1}(t)
  \psi(t)
  \, \frac{d t}{|t|}
  \, d x \, d^\times y.
  \]
  To obtain the claimed formula,
  we apply for $u \in U_1$ the substitutions
  $t \mapsto u t, x \mapsto x/u$, invoke the assumed
  $U_1$-invariance of $\overline{f} \otimes f$, and average over $u$.
\end{proof}

\subsection{Gauss sums}
\label{sec-33-6}
We shall repeatedly use the following without explicit mention:
\begin{lemma}\label{lem:gauss-sums}
  Let $U_1 \leq \mathfrak{o}^\times$ be an open subgroup
  and
  $\omega$ a character of $\mathfrak{o}^\times$.
  For $t \in k^\times$,
  set $H(t) := H(t,\omega,U_1) := \mathbb{E}_{u \in U_1}
  \omega(u t) \psi(u t)$, where $\mathbb{E}$ denotes
  integration with respect to the probability Haar.
  \begin{enumerate}
  \item For fixed $U_1$, one has
    $H(t) = 0$ unless
    $-\nu(t) = c(\omega) + O(1)$, in which case
    $H(t) \ll C(\omega)^{-1/2}$, with implied constants
    depending at most upon $U_1$.
  \item Suppose $U_1 = \mathfrak{o}^\times$
    and $c(\omega) > 0$.
    Then $H(t) = 0$ unless $- \nu(t) = c(\omega)$,
    in which case $H(t)$ is independent of $t$
    and has magnitude $|H(t)| = c C(\omega)^{-1/2}$ for some $c > 0$
    depending
    only upon $k$.
  \end{enumerate}
\end{lemma}
\begin{proof}
  For $U_1 = \mathfrak{o}^\times$, these are standard assertions concerning
  Gauss sums.  The standard proof adapts to the general case
  (compare with \cite[3.1.14]{michel-2009}).
\end{proof}

\section{Local study of non-archimedean microlocal lifts}
\label{sec:basic-theory-microlocal}
Recall Definition \ref{defn:microlocal-lifts} and
the statement of Lemma \ref{thm:compute-microlocal-lifts}.
Retain the notation of \S\ref{sec-33}.

\subsection{Proof of Lemma \ref{thm:compute-microlocal-lifts}: determination of microlocal lifts}\label{sec:proof-determination-microlocal}
  For any character $\chi : k^\times \rightarrow \mathbb{C}^\times$, the non-equivariant twisting
  isomorphism
  $\pi \rightarrow \pi' := \pi \otimes \chi$
  induces non-equivariant linear isomorphisms
  \begin{align}\label{eqn:twisting}
    V &:= \{ \text{ microlocal lifts in $\pi$
      of orientation $(\omega_1,\omega_2)$ } \} \\ \nonumber
    &\quad\cong
    \{ \text{ microlocal lifts in $\pi'$
      of orientation $(\omega_1',\omega_2')$ } \}
  \end{align}
  with $\omega_i' := \omega_i \cdot \chi|_{\mathfrak{o}^\times}$.
  We thereby reduce to verifying
  the conclusion in the special case $\omega_1 = 1$.
  Suppose $V \neq 0$.
  Write $\omega := \omega_2$.
  By the convention $N \geq 1$ of Definition \ref{defn:microlocal-lifts},
  $\omega$ is ramified.
  The central character $\chi_\pi$ of $\pi$
  restricts to $\omega$,
  hence is ramified;
  by Theorem \ref{thm:basic-local-newvectors}, $\dim V = \max(0,1 + c(\pi) -
  c(\chi_\pi))$,
  and so $V \neq 0$ only if $c(\pi) \geq c(\chi_\pi)$.
  By Lemma \ref{lem:extremal-central-chars}, the latter happens
  only if $c(\pi) = c(\chi_\pi)$
  and $\pi$ has the indicated form, in which case $\dim V = 1$.
  The explicit description of $V$ now follows in general from
  \eqref{eqn:twisting}.
\subsection{Explicit formulas}
\label{sec-explicit-formulas}
  Let $\pi := \chi_1 \boxplus \chi_2$ and $\omega_i :=
  \chi_i|_{\mathfrak{o}^\times}$
  with $N := c(\omega_1/\omega_2) \geq 1$.
\begin{lemma}
  Define $f_1, f_2 \in C^\infty(k)$
  (as if in the ``line model'' of \S\ref{sec-33-2})
  by
  \[
  f_1(x) := 1_{\mathfrak{p}^{N_2}}(x),
  \quad
  f_2(x) := 1_{\mathfrak{p}^{N_1}}(1/x) |1/x| \chi_1^{-1} \chi_2(x)
  \]
  and
  $v_1, v_2 \in \pi$ in the induced model
    on $g = \begin{pmatrix}
      \ast & \ast \\
      c & d
    \end{pmatrix} \in \GL_2(k)$ by
    \begin{align}\label{eq:induced-formula-v1}
      v_1 (g) &:= v_{f_1}(g) =
                1_{\mathfrak{p}^{N_2}}(c/d)
                \left\lvert \frac{\det g}{d^2} \right\rvert^{1/2}
                \chi_1(\det (g)/ d) \chi_2(d), \\ \label{eq:induced-formula-v2}
      v_2 (g) &:= v_{f_2}(g) =
                1_{\mathfrak{p}^{N_1}}(d/c)
                \left\lvert \frac{\det g}{c^2} \right\rvert^{1/2}
                \chi_1(\det (g)/ c) \chi_2(c)
    \end{align}
    and
   $W_1, W_2 \in \pi$ in the Kirillov
    model
    $\mathcal{K}(\pi,\psi)$ 
    by\footnote{Recall that $\psi$ is assumed unramified.}
    \begin{align}\label{eq:microlocal-kirillov-formula-w1}
      W_1(y) &:=
               1_{\mathfrak{p}^{-N_1}}(y)
               |y|^{1/2}
               \chi_1(y),
               \quad 
      W_2(y) :=
               1_{\mathfrak{p}^{-N_1}}(y)
               |y|^{1/2}
               \chi_2(y).
    \end{align}
    Then $v_1, W_1$ and $v_2, W_2$
    are microlocal lifts
    of orientations $(\omega_1,\omega_2)$
    and $(\omega_2,\omega_1)$, respectively.
\end{lemma}
\begin{proof}
  The formulas
  for $W_1, v_1$
  in the case $\chi_1 = 1$
  and those for $W_2, v_2$
  in the case $\chi_2 = 1$
  follow from known formulas
  for standard newvectors \cite{Sch02}; the general case
  follows from the twisting isomorphisms (\ref{eqn:twisting}).
\end{proof}





\subsection{Stationary phase analysis of local Rankin--Selberg integrals}
\label{sec-stationary-phase-local-rs}
In this section we apply stationary phase analysis to evaluate and estimate
some local Rankin--Selberg integrals involving microlocal lifts
and newvectors.
We use these in \S\ref{sec:microlocal-basics}
to prove
Theorem \ref{thm:microlocal-basics}
and
Theorem
\ref{thm:weakly-subconvex}.
Retain the notation of \S\ref{sec-33-1}.  Let $\chi_1, \chi_2$ be
unitary characters of $k^\times$ for which
$N := c(\chi_1/\chi_2)$ is positive.  Let
$\pi = \chi_1 \boxplus \chi_2$ be the corresponding generic
irreducible unitary principal series representation of
$\GL_2(k)$, realized in its induced model and equipped with the
norm given in \S\ref{sec-33-2}.  Equip the complex-conjugate representation
$\overline{\pi}$ with the compatible unitary structure.  Define
the intertwiner
$\pi \ni v \mapsto W_v \in \mathcal{W}(\pi,\psi)$ as in \S\ref{sec-33-3}.
Let $\sigma$ be a generic irreducible unitary representation
of $\PGL_2(k)$,
realized in its $\psi$-Whittaker model
$\sigma = \mathcal{W}(\sigma,\psi)$.
\begin{theorem}\label{thm:local-rs}
  Let $v \in \pi$ be a microlocal lift of orientation
  $(\chi_1|_{\mathfrak{o}^\times},
  \chi_2|_{\mathfrak{o}^\times})$,
  let $v' \in \pi$ be a generalized newvector,
  and let $W_1 \in \sigma$.
  \begin{enumerate}
  \item[(I)]
    If $N$ is large enough
    in terms of $W_1$,
    then
    \[
    \ell_{\RS}(W_1, \overline{W_v}, v)
    = c q^{-N/2}
    \|v\|^2 \int_{y \in k^\times} W_1(y) \, d^\times y
    \]
    where\footnote{
      The integral defining $c$ should be interpreted
      in the usual way
      as (for instance)
      a limit of integrals over increasing finite unions of $\mathfrak{o}^\times$-cosets.}
    $c := q^{N/2} \int_{t \in k^\times} \chi_1 \chi_2^{-1}(t) \psi(t) \,
    \frac{d t}{|t|} \asymp 1$
    is a complex scalar which is independent of $W_1$ and whose
    magnitude depends only upon $k$.
  \item[(II)] One has
    $\ell_{\RS}(W_1 \otimes \overline{W_{v'}} \otimes v') \ll
    q^{-N/2} \|v'\|^2$
    with the implied constant depending at most upon $W_1$.
  \item[(III)] Suppose that $\nu(2) = 0$,
    $\chi_\pi$ is unramified,
    $\|v'\| = \|v\|$,
    the support of $v'$ is $-N..N$,
    $\sigma$ is unramified,
    and $W_1 \in \sigma$ is spherical,
    so that $N = c(\chi_1) = c(\chi_2)$
    and $c(\pi) = 2 N$.
    Then
    $\ell_{\RS}(W_1, \overline{W_v}, v)
    = \ell_{\RS}(W_1, \overline{W_{v'}}, v')$.
  \end{enumerate}
\end{theorem}
The most difficult assertion is (II), which is used only to
deduce the equidistribution of newvectors (Theorem
\ref{thm:equid-balanced-newvectors}).
Assertion (III)
serves only the purpose of illustration (cf. the discussion
after Theorem \ref{thm:microlocal-basics}).  The other main
results of this article (Theorems \ref{thm:weakly-subconvex},
\ref{thm:main}) require only (I), whose proof is very short.
\begin{proof}[Proof of (I)]
  Without loss of generality,
  let $v = v_f$ with $f(x) := 1_{\mathfrak{p}^{N_2}}(x)$.
  Because $N_2$ is large in enough in terms of $W_1$,
  we have whenever $f(x) \neq 0$
  that
  $W_1(a(y) n'(x/u)) = W_1(y)$ for all $u \in
  \mathfrak{o}^\times$.
  Lemma \ref{lem:appl-diag-invar}
  gives after the simplifications
  $f(x) \overline{f}(x + y/t)
  =
  1_{\mathfrak{p}^{N_2}}(x) 1_{\mathfrak{p}^{N_2}}(y/t)$
  and
  $1_{\mathfrak{p}^{N_2}}(x) F(x,y,t;W_1,\mathfrak{o}^\times)
  = 1_{\mathfrak{p}^{N_2}}(x) W_1(y) H(t)$
  with
  $H(t) := \mathbb{E}_{u \in \mathfrak{o}^\times} \chi_1
  \chi_2^{-1} (u t) \psi(ut)$
  that
  \[
  \ell_{\RS}(W_1, \overline{W_v}, v)=
  \int_{y \in k^\times}
  W_1(y)
  \int_{x \in k}
  1_{\mathfrak{p}^{N_2}}(x)
  \int_{t \in k}
  1_{\mathfrak{p}^{N_2}}(y/t)
  H(t)
  \, \frac{d t}{|t|}
  \, d x
  \, d^\times y.
  \]
  We have $W_1(y) H(t) = 0$ unless $|t| \asymp q^N$ and $|y|
  \ll 1$;
  because $N_1$ is large enough in terms of $W_1$,
  the factor $1_{\mathfrak{p}^{N_2}}(y/t) = 1$ is thus
  redundant.
  Since $\int_{x \in k} 1_{\mathfrak{p}^{N_2}}(x) \, d x =
  \int_{k} |f|^2 = \|v\|^2$, we obtain
  the required identity.
\end{proof}
\begin{proof}[Proof of (II)]
  Suppose first that $\chi_1$ and $\chi_2$ are both
  ramified.
  In that case, Lemma
  \ref{lem:balanced-newvectors-in-induced-model}
  says that $v' = v_f$ with $f$ a character multiple of
  the characteristic 
  function of some $\mathfrak{o}^\times$ coset.
  In particular:
  \begin{equation}\label{eq:properties-of-f-for-upper-bounds}
    \text{$f$ is supported on a coset of $\mathfrak{o}^\times$,
      and $\overline{f} \otimes f$ is $\mathfrak{o}^\times$-invariant.}
  \end{equation}
  From the mod-center identity
  $a(y) n'(x) \equiv n(y/x) a(y/x^{2}) w n(1/x)$,
  we have
  \begin{equation}\label{eqn:neat-identity-whittaker-function}W_1(a(y) n'(x))  = \psi(y/x) W_1(a(y/x^2) w n(1/x)).\end{equation}
From \eqref{eqn:neat-identity-whittaker-function} and standard
bounds on Whittaker functions,
we have\footnote{
  See \cite[3.2.3]{michel-2009},
  and recall that $\sigma$ is assumed generic and
  unitary.}
\begin{equation}\label{eq:uniform-integrability-whittaker-function}
  \sup_{x \in k} \int_{y \in k^\times} |W_1(a(y) n'(x))| \, d^\times y \ll 1.
\end{equation}
By \eqref{eqn:neat-identity-whittaker-function}, there exists a fixed open subgroup $U_1 \leq \mathfrak{o}^\times$
for which
\begin{equation}
  \label{eqn:whittaker-variation-wrt-u}
  W_1(a(y) n'(x/u))
  =
  W_1(a(y) n'(x))
  \times 
  \begin{cases}
    1 & \text{ for } |x| \leq 1,
    \\
    \psi((u-1) y/x)
    & \text{ for } |x| \geq 1.
  \end{cases}
\end{equation}
  Without loss of generality, suppose $\int_k |f|^2 = 1$.
  We apply Lemma \ref{lem:appl-diag-invar},
  split the integral according as $|x| \leq 1$ or not,
  and appeal to
  \eqref{eq:uniform-integrability-whittaker-function} 
  and \eqref{eqn:whittaker-variation-wrt-u};
  our task thereby reduces to showing
  with
  \begin{align*}
    H_1(t)
    &:=
      \mathbb{E}_{u \in U_1} \chi_1 \chi_2^{-1}(u t) \psi(u t),
    \\
    H_2(t,y/x)
    &:= 
      \psi(-y/x) \mathbb{E}_{u \in U_1} \chi_1 \chi_2^{-1}(u t) \psi(u (t +
      y/x))
  \end{align*}
  that the quantities
  \begin{align*}
    I_1 &:=
          \sup_{y \in k^\times} \int_{t \in k^\times}
          \int_{
          x \in k : |x| \leq 1 
    }
    \lvert f(x) \overline{f}(x + y/t)
    H_1(t) \rvert d^\times t \, d x \, d^\times y, \\
    I_2 &:=
          \sup_{y \in k^\times} \int_{t \in k^\times}
          \int_{
          x \in k : |x| > 1 
    }
    \lvert f(x) \overline{f}(x + y/t)
    H_2(t,y/x) \rvert d^\times t \, d x \, d^\times y
  \end{align*}
  are $O(q^{-N/2})$.
  We have $H_1(t) = 0$ unless $|t| \asymp q^{N}$,
  in which case $H_1(t) \ll q^{-N/2}$; the set of such $t$ has $d^\times t$-volume $O(1)$,
  so an adequate estimate for $I_1$
  follows
  from Cauchy--Schwartz applied to the $x$-integral.
  Similarly, $H_2(t,y/x) = 0$ unless
  $|t + y/x| \asymp q^N$,
  in which case $H_2(t,y/x) \ll q^{-N/2}$;
  the support condition on $f$ shows that
  $f(x) \overline{f}(x + y/t) =0$ unless
  $|t + y/x| = |t|$,
  we may conclude once again by Cauchy--Schwartz.\footnote{\label{footnote:weak-sc}
    The estimate just derived is essentially sharp when $f$ is supported in a fixed open subset of $k^\times$,
    but can be substantially sharpened
    when $f$ is ``unbalanced''
    in the sense that its support tends sufficiently rapidly with $N$ either to zero or
    infinity.
    The possibility of such sharpening is the simplest case of the ``weak
    subconvexity''
    phenomenon identified in \cite{PDN-AP-AS-que}.
  }

  We turn to the case that one of $\chi_1, \chi_2$ is
  unramified.
  By the assumption $c(\chi_1/\chi_2) \neq 0$,
  the other one is ramified.
  By symmetry, we may suppose that $\chi_1$ is unramified
  and $\chi_2$ is ramified.
  By Lemma \ref{lem:balanced-newvectors-in-induced-model}, we may suppose without loss of generality
  that $v' = v_f$ for $f =  1_{\mathfrak{a}}$
  with $\mathfrak{a} \subset
  k$
  a fractional $\mathfrak{o}$-ideal.
  Then
  $\overline{f} \otimes f$ is $\mathfrak{o}^\times$-invariant.
  We split the integral over $x \in k$ as above,
  and the same argument works for the range $|x| \leq 1$.
  The remaining
  range contributes
  \begin{equation}\label{eq:I3-defn-integral}
      I_3 :=
  \int_{x \in k: |x| > 1}
  \int_{y \in k^\times}
  \int_{t \in k^\times}
  1_\mathfrak{a}(x) 1_\mathfrak{a}(x + y/t)
  H_3(t, y/x; x)
  \, d x \, d^\times y \, d^\times t
  \end{equation}
  where
  \[
  H_3(t, y/x; x)
  :=
  W_1(a(y/x^2) w n(1/x))
  \mathbb{E}_{u \in U_1}
  \chi_1 \chi_2^{-1}(u t) \psi(u (t + y/x)).
  \]
  A bit more care is required than in the above argument, which
  gives now an upper bound of $+\infty$;
  the problem is
  that the nonvanishing of $H_3(t,y/x;x)$ no longer restricts
  $t$ to a volume $O(1)$ subset of $k^\times$.
  We do better here by
  exploiting additional cancellation coming from the
  $y$-integral:  Let $C_1, C_2$ be positive scalars, depending
  only upon $W_1, U_1$, so that
  \begin{equation}\label{eq:}
    H_3(t,y/x;x) \neq 0
    \implies
    C_1 q^N < |y/x + t| < C_2 q^N.
  \end{equation}
  If $|y/x| \geq  C_2 q^N$,
  then $H_3(t,y/x;x) \neq 0$
  only if $|t| = |y/x|$.
  If $|y/x| \leq  C_1 q^N$,
  then $H_3(t,y/x;x) \neq 0$
  only if $C_1 q^N < |t| < C_2 q^N$.
  Arguing as above, we reduce to considering the range
  $C_1 q^N < |y/x| < C_2 q^N$,
  in which
  $H_3(t,y/x;x) \neq 0$
  only if $|t| <  C_2 q^N$.
  The range $C_1 q^N \leq |t| < C_2 q^N$
  may be treated as before,
  so we reduce to showing that
  \begin{equation}\label{eq:task-I4}
    I_4 := \int_{
      \substack{
        x,y,t \in k, k^\times, k^\times :
        \\
        |x| > 1,
        \\
        C_1 q^N < |y/x| < C_2 q^N,
        \\
        |t| < C_1 q^N
      }
    }
    1_\mathfrak{a}(x) 1_\mathfrak{a}(x + y/t)
    H_3(t,y/x;x) \, d x \, d^\times y \, d^\times t = 0.
  \end{equation}
  Note that the conditions defining the integrand imply that $|x t/y| < 1$.
  There is an open subgroup $U_2$ of
  $\mathfrak{o}^\times$,
  depending only upon $W_1, U_1$,
  so that
  \begin{equation}\label{eq:some-ratio-will-be-in-U1}
    z \in U_2, |x t/y| < 1 \implies
    \frac{t + z y / x}{t + y/x} \in U_1,
  \end{equation}
  \begin{equation}\label{eq:whittaker-function-doesnt-change-much-z}
    |x| > 1, z \in U_2
    \implies
    W_1(a(z y/x^2) w n(1/x)) = W_1(a(y/x^2) w n(1/x)),
  \end{equation}
  \begin{equation}\label{eq:char-fn-frac-ideal-doesnt-change-much}
    |x t/y| < 1,
    z \in U_2
    \implies
    1_\mathfrak{a}(x + z y/t)
    =
    1_\mathfrak{a}(x + y/t).
  \end{equation}
  For $N$ large enough in terms of $W_1, U_1$ and hence $U_2$,
  we have
  \begin{equation}\label{eq:vanishing-U2-integral}
    |x t/y| < 1
    \implies
    \mathbb{E}_{z \in U_2}
    \chi_1^{-1} \chi_2(t + z y/x)
    = 0.
  \end{equation}
  In $I_4$, we substitute
  $y \mapsto y z$ with $z \in U_2$ and average over $z$;
  by \eqref{eq:char-fn-frac-ideal-doesnt-change-much}
  and \eqref{eq:whittaker-function-doesnt-change-much-z},
  our task
  reduces to establishing for $|x t/y| < 1$
  that
  \[
  \mathbb{E}_{z \in U_2}
  \mathbb{E}_{u \in U_1}
  \chi_1 \chi_2^{-1}(u t) \psi(u (t + z y/x)) = 0,
  \]
  which follows
  from
  \eqref{eq:vanishing-U2-integral}
  after
  the change of variables
  $u \mapsto u (t + y/x) / (t + z y/x)$ suggested
  by \eqref{eq:some-ratio-will-be-in-U1}.
\end{proof}  
\begin{proof}[Proof of (III)]
  By (I), our task
  reduces to showing that
  \[\ell_{\RS}(W_1, \overline{W_{v'}}, v') = c q^{-N/2} \|v'\|^2
  \int_{y \in k^\times} W_1(y) \, d^\times y\]
  with the same scalar $c$ as in (I).
  Suppose without loss of generality that
  $v' = v_f$ with $f := \chi_2 1_{\mathfrak{o}^\times}$.
  Note that $\overline{f} \otimes f$ is
  $\mathfrak{o}^\times$-invariant.
  If $f(x) \neq 0$, then
  $W_1(a(y) n'(x/u)) = W_1(y)$ for all $u
  \in \mathfrak{o}^\times$.
  Lemma \ref{lem:appl-diag-invar}
  gives
  after the simplification
  $f(x) F(x,y,t;W_1,\mathfrak{o}^\times)
  = f(x) W_1(y) H(t)$
  with
  $H$ as in the proof of (I)
  that
  \[
  \ell_{R S}(W_1, \overline{W_{v'}}, v') =
  \int_{y \in k^\times}
  \int_{x \in k}
  \int_{t \in k}
  W_1(y)
  f(x)
  \overline{f} ( x  + y/t )
  H(t)
  \, \frac{d t}{|t|}
  \, d x
  \, d^\times y.
  \]
  Because $\nu(2) = 0$,
  we have $c(\chi_2) = c(\chi_1 \chi_2^{-1}) = N$.
  Thus if $W_1(y) f(x) H(t) \neq 0$,
  then $y,x,t \in \mathfrak{o}, \mathfrak{o}^\times, \varpi^{-N}
  \mathfrak{o}^\times$
  and so
  $f(x)
  \overline{f} ( x  + y/t )= 1$.
  From $\int_{x \in k} 1_{\mathfrak{o}^\times}(x) \, d x = \int_k |f|^2 = \|v'\|^2$,
  we conclude.
\end{proof}

\begin{remark}\label{remark:mv-epic}
  \cite[3.4.2]{michel-2009} and Theorem \ref{thm:local-rs}(I) and \cite[(3.25)]{michel-2009} imply the following:
  Let $v_2, v_3 \in \pi$ be microlocal
  lifts of the same orientation
  and $v_1 \in \sigma$, realized in its Kirillov model
  $\mathcal{K}(\sigma,\psi)$.
  The formula $\|v_1\|^2 := \int_{y \in k^\times} |v_1(y)|^2 \,
  d^\times y$
  is known to define an invariant norm on $\sigma$.
  Suppose that $N$ is large enough in terms of $v_1$.
  Then
  \[
  \int_{g \in Z \backslash G}
  \prod_{i=1,2,3}
  \langle \pi_i(g) v_i, v_i \rangle
  =
  c q^{-N}
  \|v_2\|^2 \|v_3\|^2
  \int_{y \in k^\times}
  \langle a(y) v_1, v_1 \rangle \, d^\times y
  \]
  for some positive scalar $c \asymp 1$ depending only upon $k$.
  This identity solves the problem of producing a
  subconvexity-critical test vector for the local triple product
  period in the QUE case when the varying representation is
  principal series.  It would be interesting
  to verify whether the supercuspidal case follows
  similarly using a modification of Definition
  \ref{defn:microlocal-lifts} involving characters on an
  $\eps$-neighborhood in $\GL_2(\mathfrak{o})$
  of the points of a suitable non-split torus, where
  $\eps \asymp C(\overline{\pi} \otimes \pi)^{-1/4}$.
\end{remark}

\section{Completion of the proof\label{sec:microlocal-basics}}
\label{sec-8}
In this section,
$\varphi  \in \pi \in A_0(\mathbf{X})$ traverses
a sequence of $L^2$-normalized
microlocal lifts on $\mathbf{X}$ of level $N \rightarrow
\infty$.
Thus $\varphi$ and $\pi$, like most objects to be considered
in this section, depend upon $N$, but we omit this dependence
from our notation.
We use the abbreviations
\emph{fixed} to mean ``independent of $N$''
and \emph{eventually} to mean ``for large enough $N$.''
Asymptotic notation such as $o(1)$ refers to the $N \rightarrow \infty$ limit.
Our aim is to verify the conclusions of
Theorem \ref{thm:microlocal-basics}
and
Theorem
\ref{thm:weakly-subconvex}.

As $G$-modules, $\pi \cong \chi_1 \boxplus \chi_2$
for some unitary characters $\chi_1, \chi_2$ of
$\mathbb{Q}_p^\times$
for which $c(\chi_1/\chi_2) = N$.

Recall our simplifying assumption that $R$ is a maximal order.
This implies that for any irreducible $\mathcal{H}$-submodule $\pi'$ of $\mathcal{A}(\mathbf{X})$,
the vector space underlying $\pi'$ is an irreducible admissible $G$-module.
In other words, the local components
at all places $v \neq p$ are one-dimensional.

The function $\varphi$ has unitary central character,
so the measure $\mu_\varphi$ is invariant by the center.
Moreover, for each prime $\ell \mid \disc(B)$,
the involution $T_{\ell}$ acts on $\pi$ with some eigenvalue
$\pm 1$,
hence $\mu_\varphi$ is $T_{\ell}$-invariant.
The natural space of observables against which it suffices to test
$\mu_\varphi$
is thus
\[
\mathcal{A}^+(\mathbf{X}) := \left\{
  \Psi \in \mathcal{A}(\mathbf{X}) :
  \begin{array}{lr}
    T_{\ell} \Psi = \Psi \text{ for } \ell \mid \disc(B), \\
    z \Psi =
    \Psi \text{ for } z \in
    Z := \text{center of } G 
  \end{array}
\right\}.
\]
That space decomposes further as
$\mathcal{A}^+(\mathbf{X})
= (\oplus_{\chi} \mathbb{C} (\chi \circ \det)) \oplus
\mathcal{A}_0^+(\mathbf{X})$
where
\begin{itemize}
\item  $\chi$ traverses the set of quadratic characters
  of the compact group $\mathbb{Q}_p^\times /
  \mathbb{Z}[1/p]^\times$
  satisfying $\chi(\ell)= 1$ for $\ell \mid \disc(B)$, and
\item
  $\mathcal{A}_0^+(\mathbf{X}) := \mathcal{A}^+(\mathbf{X}) \cap
  \mathcal{A}_0(\mathbf{X})$,
  which decomposes further as a countable direct sum
  $\mathcal{A}_0^+(\mathbf{X}) = \oplus_{\sigma \in
    A_0^+(\mathbf{X})}\sigma$
  where we substitute $A$ for $\mathcal{A}$ to denote
  ``irreducible submodules of.''
\end{itemize}

Let $\sigma \in A^+(\mathbf{X})$ be fixed.
It is either one-dimensional and of the form
$\mathbb{C} (\chi \circ \det)$ for some $\chi$ as above,
or belongs to
$A_0^+(\mathbf{X})$ and is generic as a $G$-module.  Denote
by
$\ell : \sigma \otimes \overline{\pi } \otimes \pi \rightarrow
\mathbb{C}$
the $G$-invariant functional defined by integration over
$\mathbf{X}$.
\begin{lemma}\label{lem:one-dimensionals}
  Suppose $\sigma$ is
  one-dimensional
  and $\ell \neq 0$.
  Then $\sigma$ is trivial eventually.
\end{lemma}
\begin{proof}
  Write $\sigma = \mathbb{C} (\chi \circ \det)$
  for some quadratic character $\chi$.  By Schur's lemma,
  $\pi \cong \chi_1 \boxplus \chi_2$ is isomorphic
  as a $G$-module to $\pi \otimes \chi \circ \det \cong \chi_1 \chi \boxplus \chi_2 \chi$,
  which is known to happen
  only if either $\chi_1 = \chi_1 \chi$, in which case
  $\chi$ is trivial, or $\chi_1 = \chi_2 \chi$, in which
  case
  $c(\chi) = c(\chi_1/\chi_2) = N \rightarrow \infty$, which does
  not happen because $\chi$ is quadratic.\footnote{We
    use here that the local field $\mathbb{Q}_p$
    is not a function field of characteristic $2$.
  }
\end{proof}

We now prove
Theorem \ref{thm:microlocal-basics}.  It suffices to verify that the
various assertions hold for fixed $\Psi \in \sigma \in A^+(\mathbf{X})$.  They are tautological if
$\sigma$ is trivial, so by Lemma \ref{lem:one-dimensionals}, we reduce to
the case that $\sigma \in A_0^+(\mathbf{X})$ is generic.
Fix an unramified non-trivial character $\psi : \mathbb{Q}_p
\rightarrow \mathbb{C}^{(1)}$
and $G$-equivariant isometric isomorphisms $\sigma \cong
\mathcal{W}(\sigma,\psi)$,
$\pi \cong \chi_1 \boxtimes \chi_2$.
Denote by $\ell_{\RS} : \sigma \otimes \overline{\pi }
\otimes \pi
\rightarrow \mathbb{C}$ the trilinear form
defined in \S\ref{sec-33-5}.
By Theorem \ref{thm:uniqueness-trilinear-functionals}
and the nonvanishing of $\ell_{\RS}$,
there exists a complex scalar $\mathcal{L}^{1/2} \in
\mathbb{C}$
so that
\begin{equation}\label{eq:uniq-tril-short}
  \ell = \mathcal{L}^{1/2} \ell_{\RS}.
\end{equation}
Theorem \ref{thm:local-rs}(I) implies that
$\ell_{\RS}(\sigma(a(y)) \Psi, \overline{\varphi}, \varphi) = \ell_{\RS}(\Psi, \overline{\varphi}, \varphi)$ holds eventually
for fixed $y \in k^\times$; the required diagonal invariance
then follows from \eqref{eq:uniq-tril-short}.
If $p \neq 2$ and $\varphi '$ is an $L^2$-normalized newvector of
support $-N..N$
and $\Psi \in \sigma^K$ is spherical, then
Theorem \ref{thm:local-rs}(III)
gives $\ell_{\RS}(\Psi, \overline{\varphi}, \varphi)
= \ell_{\RS}(\Psi, \overline{\varphi'}, \varphi')$ eventually; the
lifting property
then follows from \eqref{eq:uniq-tril-short}.
For the equidistribution application,
we reduce by
Lemma \ref{lem:one-dimensionals}
and
\eqref {eq:uniq-tril-short}
and
Theorem \ref{thm:local-rs}(II)
to
showing that $\mathcal{L}^{1/2} =
o(p^{N/2})$ holds
under the hypothesis that for each fixed $\Psi_0 \in \sigma$,
one has $\ell(\Psi_0,\overline{\varphi}, \varphi) = o(1)$.
Let $\Psi_0 \in \sigma \cong \mathcal{W}(\sigma,\psi)$
be given in the Kirillov model by the characteristic function
of the unit group.
By
Theorem \ref{thm:local-rs}(I),
$\ell_{\RS}(\Psi_0,\overline{\varphi}, \varphi) \asymp
p^{-N/2}$ eventually, 
so our hypothesis and \eqref{eq:uniq-tril-short} give the required estimate for $\mathcal{L}^{1/2}$.
 
We turn to the proof of Theorem \ref{thm:weakly-subconvex}.
Our assumptions on $\pi$ and $\sigma$
imply that $\sigma \in A_0^+(\mathbf{X})$ and that the
adelizations of $\sigma, \overline{\pi}$ and $\pi$
at each $v \in S_B := \{\infty \} \cup \{\ell : \ell \mid
\disc(B)\}$
are one-dimensional
and have trivial tensor product, hence that the product of their 
normalized matrix coefficients is one;
by Ichino's formula \cite{MR2585578} and
\cite[3.4.2]{michel-2009},
it follows that
$L \asymp |\mathcal{L}^{1/2}|^2$,
where $L$ denotes the LHS of \eqref{eq:weakly-subconvex}
and $\mathcal{L}^{1/2}$ is as above (compare with Remark \ref{remark:mv-epic}).
By Theorem \ref{thm:main} and the argument of the previous paragraph,
$\mathcal{L}^{1/2} = o(p^{N/2})$.
Our goal is to show that $L = o(C^{1/4})$,
where $C := C(\sigma \times \overline{\pi} \times
\pi)$ is the global conductor;
the contribution to $C$ from $v \in S_B$ is bounded,
hence
$C
\asymp
C(\sigma_p \otimes \chi_1^{-1} \chi_2)
C(\sigma_p \otimes \chi_2^{-1} \chi_1)
C(\sigma_p)^2
\asymp
C(\chi_1^{-1} \chi_2)^4
=
p^{4 N}$.
The known estimate $\mathcal{L}^{1/2} = o(p^{N/2})$ thus translates
to the goal $L = o(C^{1/4})$, as required.

\bibliography{refs}{}

\def\cprime{$'$} \def\cprime{$'$} \def\cprime{$'$} \def\cprime{$'$}
\begin{thebibliography}{10}

\bibitem{MR3322309}
Nalini Anantharaman and Etienne Le~Masson.
\newblock Quantum ergodicity on large regular graphs.
\newblock {\em Duke Math. J.}, 164(4):723--765, 2015.

\bibitem{MR1957735}
Jean Bourgain and Elon Lindenstrauss.
\newblock Entropy of quantum limits.
\newblock {\em Comm. Math. Phys.}, 233(1):153--171, 2003.

\bibitem{MR0017775}
H.~Brandt.
\newblock Zur {Z}ahlentheorie der {Q}uaternionen.
\newblock {\em Jber. Deutsch. Math. Verein.}, 53:23--57, 1943.

\bibitem{MR2677974}
Shimon Brooks and Elon Lindenstrauss.
\newblock Graph eigenfunctions and quantum unique ergodicity.
\newblock {\em C. R. Math. Acad. Sci. Paris}, 348(15-16):829--834, 2010.

\bibitem{MR3038543}
Shimon Brooks and Elon Lindenstrauss.
\newblock Non-localization of eigenfunctions on large regular graphs.
\newblock {\em Israel J. Math.}, 193(1):1--14, 2013.

\bibitem{MR3260861}
Shimon Brooks and Elon Lindenstrauss.
\newblock Joint quasimodes, positive entropy, and quantum unique ergodicity.
\newblock {\em Invent. Math.}, 198(1):219--259, 2014.

\bibitem{MR1431508}
Daniel Bump.
\newblock {\em Automorphic Forms and Representations}, volume~55 of {\em
  Cambridge Studies in Advanced Mathematics}.
\newblock Cambridge University Press, Cambridge, 1997.

\bibitem{MR0337789}
William Casselman.
\newblock On some results of {A}tkin and {L}ehner.
\newblock {\em Math. Ann.}, 201:301--314, 1973.

\bibitem{MR0338274}
William Casselman.
\newblock The restriction of a representation of {${\rm GL}_{2}(k)$} to {${\rm
  GL}_{2}({\mathfrak{o}})$}.
\newblock {\em Math. Ann.}, 206:311--318, 1973.

\bibitem{MR0080767}
Martin Eichler.
\newblock Zur {Z}ahlentheorie der {Q}uaternionen-{A}lgebren.
\newblock {\em J. Reine Angew. Math.}, 195:127--151 (1956), 1955.

\bibitem{MR2366231}
Manfred Einsiedler and Elon Lindenstrauss.
\newblock On measures invariant under diagonalizable actions: the rank-one case
  and the general low-entropy method.
\newblock {\em J. Mod. Dyn.}, 2(1):83--128, 2008.

\bibitem{MR894322}
Benedict~H. Gross.
\newblock Heights and the special values of {$L$}-series.
\newblock In {\em Number theory ({M}ontreal, {Q}ue., 1985)}, volume~7 of {\em
  CMS Conf. Proc.}, pages 115--187. Amer. Math. Soc., Providence, RI, 1987.

\bibitem{MR2680499}
Roman Holowinsky and Kannan Soundararajan.
\newblock Mass equidistribution for {H}ecke eigenforms.
\newblock {\em Ann. of Math. (2)}, 172(2):1517--1528, 2010.

\bibitem{2014arXiv1409.8173H}
Y.~{Hu}.
\newblock {Triple product formula and mass equidistribution on modular curves
  of level N}.
\newblock {\em ArXiv e-prints}, September 2014.

\bibitem{MR2449948}
Atsushi Ichino.
\newblock Trilinear forms and the central values of triple product
  {$L$}-functions.
\newblock {\em Duke Math. J.}, 145(2):281--307, 2008.

\bibitem{MR2585578}
Atsushi Ichino and Tamutsu Ikeda.
\newblock On the periods of automorphic forms on special orthogonal groups and
  the {G}ross-{P}rasad conjecture.
\newblock {\em Geom. Funct. Anal.}, 19(5):1378--1425, 2010.

\bibitem{MR3245884}
Etienne Le~Masson.
\newblock Pseudo-differential calculus on homogeneous trees.
\newblock {\em Ann. Henri Poincar\'e}, 15(9):1697--1732, 2014.

\bibitem{MR1859345}
Elon Lindenstrauss.
\newblock On quantum unique ergodicity for {$\Gamma\backslash\Bbb H\times\Bbb
  H$}.
\newblock {\em Internat. Math. Res. Notices}, (17):913--933, 2001.

\bibitem{MR2459293}
Elon Lindenstrauss.
\newblock Adelic dynamics and arithmetic quantum unique ergodicity.
\newblock In {\em Current developments in mathematics, 2004}, pages 111--139.
  Int. Press, Somerville, MA, 2006.

\bibitem{MR2195133}
Elon Lindenstrauss.
\newblock Invariant measures and arithmetic quantum unique ergodicity.
\newblock {\em Ann. of Math. (2)}, 163(1):165--219, 2006.

\bibitem{MR963118}
A.~Lubotzky, R.~Phillips, and P.~Sarnak.
\newblock Ramanujan graphs.
\newblock {\em Combinatorica}, 8(3):261--277, 1988.

\bibitem{michel-2009}
Philippe Michel and Akshay Venkatesh.
\newblock The subconvexity problem for {${\rm GL}_2$}.
\newblock {\em Publ. Math. Inst. Hautes \'Etudes Sci.}, (111):171--271, 2010.

\bibitem{PDN-HQUE-LEVEL}
Paul~D. Nelson.
\newblock Equidistribution of cusp forms in the level aspect.
\newblock {\em Duke Math. J.}, 160(3):467--501, 2011.

\bibitem{PDN-HMQUE}
Paul~D. Nelson.
\newblock Mass equidistribution of {H}ilbert modular eigenforms.
\newblock {\em The Ramanujan Journal}, 27:235--284, 2012.

\bibitem{nelson-variance-73-2}
Paul~{D}. Nelson.
\newblock Quantum variance on quaternion algebras, {I}.
\newblock preprint, 2016.

\bibitem{PDN-AP-AS-que}
Paul~D. Nelson, Ameya Pitale, and Abhishek Saha.
\newblock Bounds for {R}ankin--{S}elberg integrals and quantum unique
  ergodicity for powerful levels.
\newblock {\em J. Amer. Math. Soc.}, 27(1):147--191, 2014.

\bibitem{MR911357}
I.~Piatetski-Shapiro and Stephen Rallis.
\newblock Rankin triple {$L$} functions.
\newblock {\em Compositio Math.}, 64(1):31--115, 1987.

\bibitem{MR579066}
Arnold Pizer.
\newblock An algorithm for computing modular forms on {$\Gamma _{0}(N)$}.
\newblock {\em J. Algebra}, 64(2):340--390, 1980.

\bibitem{MR1059954}
Dipendra Prasad.
\newblock Trilinear forms for representations of {${\rm GL}(2)$} and local
  {$\epsilon$}-factors.
\newblock {\em Compositio Math.}, 75(1):1--46, 1990.

\bibitem{MR1859598}
Andre Reznikov.
\newblock Laplace-{B}eltrami operator on a {R}iemann surface and
  equidistribution of measures.
\newblock {\em Comm. Math. Phys.}, 222(2):249--267, 2001.

\bibitem{MR1266075}
Ze{\'e}v Rudnick and Peter Sarnak.
\newblock The behaviour of eigenstates of arithmetic hyperbolic manifolds.
\newblock {\em Comm. Math. Phys.}, 161(1):195--213, 1994.

\bibitem{sarnak-progress-que}
Peter Sarnak.
\newblock Recent {P}rogress on {Q}{U}{E}.
\newblock \url{http://www.math.princeton.edu/sarnak/SarnakQUE.pdf}, 2009.

\bibitem{Sch02}
Ralf Schmidt.
\newblock Some remarks on local newforms for {$\rm GL(2)$}.
\newblock {\em J. Ramanujan Math. Soc.}, 17(2):115--147, 2002.

\bibitem{MR1954121}
Jean-Pierre Serre.
\newblock {\em Trees}.
\newblock Springer Monographs in Mathematics. Springer-Verlag, Berlin, 2003.
\newblock Translated from the French original by John Stillwell, Corrected 2nd
  printing of the 1980 English translation.

\bibitem{SV-AQUE}
Lior Silberman and Akshay Venkatesh.
\newblock Quantum unique ergodicity for locally symmetric spaces {I}{I}.

\bibitem{MR2346281}
Lior Silberman and Akshay Venkatesh.
\newblock On quantum unique ergodicity for locally symmetric spaces.
\newblock {\em Geom. Funct. Anal.}, 17(3):960--998, 2007.

\bibitem{soundararajan-2008}
Kannan Soundararajan.
\newblock Weak subconvexity for central values of {$L$}-functions.
\newblock {\em Ann. of Math. (2)}, 172(2):1469--1498, 2010.

\bibitem{sage2015}
W.\thinspace{}A. Stein et~al.
\newblock {\em {S}age {M}athematics {S}oftware ({V}ersion 6.7)}.
\newblock The Sage Development Team, 2015.
\newblock {\tt http://www.sagemath.org}.

\bibitem{MR3272013}
Nicolas Templier.
\newblock Large values of modular forms.
\newblock {\em Camb. J. Math.}, 2(1):91--116, 2014.

\bibitem{venkatesh-2005}
Akshay Venkatesh.
\newblock Sparse equidistribution problems, period bounds and subconvexity.
\newblock {\em Ann. of Math. (2)}, 172(2):989--1094, 2010.

\bibitem{MR580949}
Marie-France Vign{\'e}ras.
\newblock {\em Arithm\'etique des alg\`ebres de quaternions}, volume 800 of
  {\em Lecture Notes in Mathematics}.
\newblock Springer, Berlin, 1980.

\bibitem{MR1814849}
Scott~A. Wolpert.
\newblock The modulus of continuity for {$\Gamma_0(m)\backslash{\Bbb H}$}
  semi-classical limits.
\newblock {\em Comm. Math. Phys.}, 216(2):313--323, 2001.

\bibitem{MR916129}
Steven Zelditch.
\newblock Uniform distribution of eigenfunctions on compact hyperbolic
  surfaces.
\newblock {\em Duke Math. J.}, 55(4):919--941, 1987.

\bibitem{MR1183602}
Steven Zelditch.
\newblock On a ``quantum chaos'' theorem of {R}. {S}chrader and {M}. {T}aylor.
\newblock {\em J. Funct. Anal.}, 109(1):1--21, 1992.

\end{thebibliography}
\bibliographystyle{plain}
\end{document}